\documentclass[10pt]{article}
\usepackage{amssymb}
\usepackage[latin1]{inputenc}
\usepackage{graphicx,color}
\usepackage{tikz}
\usepackage{calrsfs}
\def\nd{\noindent}

\usepackage[percent]{overpic}
\usepackage{epsfig}
\usepackage{color}
\newtheorem{theorem}{Theorem}[section]

\newtheorem{definition}{Definition}[section]
\newtheorem{lemma}{Lemma}[section]
\newtheorem{proposition}{Proposition}[section]

\newcommand{\fim}{\hfill\rule{2mm}{2mm}}
\newcommand{\Takac}{Tak\'a\v{c}}

\begin{document}
\title{
\vspace{0.5in} \textbf{ Multiplicity for a  strongly singular quasilinear problem via bifurcation theory}}

\author{
{\bf\large Carlos Alberto Santos$^{a,}$}\footnote{Carlos Alberto Santos acknowledges
the support of CNPq/Brazil Proc.  $N^o$ $311562/2020-5$.}\,\, ~~~~~~ {\bf\large Jacques Giacomoni $^b$} \\
~~~~~~ {\bf\large Lais Santos$^c$} 
\hspace{2mm}\\
{\it\small $^a$Universidade de Bras\'ilia, Departamento de Matem\'atica}\\
{\it\small   70910-900, Bras\'ilia - DF - Brazil}\\
{\it\small $^b$Universit\'e  de Pau et des Pays de l'Adour, LMAP (UMR E2S-UPPA CNRS 5142) }\\ 
{\it\small Bat. IPRA, Avenue de l'Universit\'e F-64013 Pau, France}\\
{\it\small $^c$Universidade de Vi\c cosa, Departamento de Matem\'atica}\\
{\it\small   CEP 36570.000, Vi\c cosa - MG, Brazil}\\
{\it\small e-mails: csantos@unb.br, jacques.giacomoni@univ-pau.fr, 
lais.msantos@ufv.br }\vspace{1mm}\\
}

\date{}
\maketitle \vspace{-0.2cm}

\begin{abstract}
A $p$-Laplacian elliptic problem in the presence of both  strongly singular and $(p-1)$-superlinear nonlinearities is considered. We employ bifurcation theory, approximation techniques and sub-supersolution method to establish the existence of an unbounded branch  of positive solutions, which is bounded in positive $\lambda-$direction and bifurcates from infinity at $\lambda=0$. As consequence of the bifurcation result, we determine intervals of existence, nonexistence and, in particular cases,  global multiplicity.

\end{abstract}

\nd {\it \footnotesize 2010 Mathematics Subject Classifications:} {\scriptsize 35J25, 35J62, 35J75, 35J92}\\
\nd {\it \footnotesize Key words}: {\scriptsize Strongly singular nonlinearities, $(p-1)$-superlinear terms, Bifurcation theory, Comparison Principle for $W_{\mathrm{loc}}^{1,p}(\Omega)$-sub and supersolutions. }

\section{Introduction}
\def\theequation{1.\arabic{equation}}\makeatother
\setcounter{equation}{0}

In this paper, we investigate the existence of an unbounded connected branch of solutions for the following $\lambda$-parameter problem
$$
 (P)~~\left\{
\begin{array}{l}
-\Delta_pu = {\lambda }\left(u^{-\delta} + u^{q}\right)  ~ \mbox{in } \Omega,\\
    u>0 ~ \mbox{in }\Omega,~~
    u=0  ~ \mbox{on }\partial\Omega,
\end{array}
\right.
$$
where $ \Omega \subset \mathbb{R}^N ( N \geq 2)$ is a smooth bounded domain, $\Delta_pu = \mbox{div}(|\nabla u|^{p-2}\nabla u)$ is the $p-$Laplacian operator, $1 < p < \infty$,  $ q > p-1 , \delta> 0$ and  $\lambda >0$ is a real parameter.

As a consequence of the singular nature at zero of the source term considered in $(P)$, the solutions of $ (P)$ are not smooth up to the boundary (see Theorem 2 in \cite{{MR1037213}}). In fact, for $\delta \geq (2p-1)/(p-1)$ the gradient blows up near the boundary in such a way that no solution belongs to $W_0^{1,p}(\Omega)$  (see Corollary 3.6 in \cite{Moh}), but  $u^\gamma\in W^{1,p}_0(\Omega)$ for some $ \gamma>\gamma(\delta)>1$. The interested reader can consult \cite{Can} and \cite{JKS} for more details about a optimal $\gamma(\delta)$. For this reason,  we adopt the following definition.

\begin{definition}\label{D2} We say that $u$ is a solution for $(P)$ if $ u \in  W_{\mathrm{loc}}^{1,p}(\Omega) \cap C_0(\overline{\Omega})$,
 $u > 0$ in $\Omega$, $(u-\varepsilon)^+  \in W_{0}^{1,p}(\Omega)$ for any $\varepsilon>0$, and
\begin{equation}\label{1}
\displaystyle\int_{\Omega}|\nabla u|^{p-2}\nabla u \nabla \varphi dx = \lambda \displaystyle\int_{\Omega} 
(u^{-\delta} + u^{q})\varphi dx ~~\mbox{for all} ~ \varphi \in C_c^{\infty}(\Omega).
\end{equation}
\end{definition}

Singular problems appear in non-newtonian fluids models, turbulent flows in porous media, glaciology and many other contexts and have been widely investigated since the remarkable work of Crandall, Rabinowitz and Tartar \cite{MR0427826} (see also reviews on the subject \cite{GhRa} and \cite{HM}). Concerning problems like 
\begin{equation}\label{eq2}
\left\{\begin{array}{l}
-\Delta_p u = \lambda u^{-\delta} + \mu u^q, ~\mbox{in} ~\Omega ,\\
u|_{\partial \Omega} = 0, ~u>0~\mbox{in} ~\Omega,\\
\delta > 0, ~ q > p -1,
\end{array}\right. 
\end{equation}
with singular terms combined with $(p-1)$-superlinear ones, we can quote the pioneer work of Coclite and Palmeri \cite{Coc}, in which the authors considered (\ref{eq2}) with $p=2$, $\delta > 0, ~q \geq 1,  ~\lambda =1$ and proved existence of $\mu_* > 0$ such that the problem (\ref{eq2}) has classical solution for $0 < \mu< \mu_*$ and has no solution for  $\mu > \mu_*$. 
Mainly highlighting the studies in which strongly singular problems ($\delta > 1$) were considered, we can also mention Hirano et. al in \cite{Hirano} and Arcoya and M\'erida \cite{Merida}. In \cite{Hirano}, the authors explored (\ref{eq2}) with $\lambda = 1$, $p=2$,  $1<q \leq 2^*-1$ and using non-smooth analysis tools proved the existence of $\mu_* > 0$ such that (\ref{eq2}) has at least two weak solutions for $0 < \mu < \mu_*$, at least one solution for $\mu = \mu_*$ and no solution for $\mu > \mu_*$. Whereas in \cite{Merida} the problem (\ref{eq2}) was studied with $p=2, ~\mu=1, 1<q ~< 2^*-1$ and a local multiplicity of $W^{1,2}_{\mathrm{loc}}-$solutions was established through penalization arguments, a priori estimates and continuation theorem of Leray-Schauder.

In the case $p \neq 2$, few are known about (\ref{eq2}), especially when $\delta >1$.  In \cite{Giaco}, the  problem (\ref{eq2}) was considered with $\mu = 1, ~0 < \delta < 1, ~p-1 < q \leq p^* - 1$ and global multiplicity with respect to the parameter $\lambda$ was established by combining  Brezis-Niremberg type result with sub-supersolution ones. Recently, Bal and Garain in \cite{Bal} generalized the results of Arcoya and M\'erida \cite{Merida} by considering $(2N+2)/(N+2) < p < N$. 

The main goal of this paper is to study $(P)$ by combining  bifurcation theory, comparison principle for sub-supersolutions in $W^{1,p}_{loc}(\Omega)$ with approximations arguments to prove the existence of a global unbounded connected of solutions for the problem $(P)$. The main advantage of this approach is that, in addition to establish multiplicity of solutions, we obtain a global connected branch of solutions of $(P)$. In the environment of singular problems, this kind of approach was considered in  \cite{Adi} and \cite{Bog}, where an analytic globally path connected branch of solutions was obtained for the cases $p=2$ and fractional Laplace operator, respectively. In these works, analytic bifurcation theory is used and requires to deal with analytic operators. In the case of nonlinear diffusion operators as $-\Delta_p$, it is not clear if it is true.

Before presenting our main results, let us set some terminologies and notations. Denoting by
$S = \{(\lambda, u) \in \mathbb{R} \times C_0(\overline{\Omega}) : ~(\lambda, u) ~\mbox{solves} ~(P) ~\mbox{in the sense of Definition \ref{D2}} \}$ and $\Sigma \subset S$ an unbounded connected set of $S$, we say that $\lambda \in \mathbb{R}$ is a bifurcation value of $\Sigma$ from infinity if there exists a sequence $\{(\lambda_n, u_n)\}_{n=1}^{\infty}  \subset 
\Sigma$ such that $\lambda_n \rightarrow \lambda$ and $\|u_n\|_{\infty} \rightarrow \infty$ as $n \rightarrow \infty$.

The normalized positive eigenfunction associated to the first eigenvalue of $(-\Delta_p, W_0^{1,p}(\Omega))$ is denoted by $\phi_1 \in C_0^1(\overline{\Omega})$,  that is,
$$-\Delta_p \phi_1 = \lambda_1\phi_1^{p-1} ~~ \mbox{in} ~\Omega, ~\phi_1|_{\partial \Omega} = 0.
$$

Now we can state our main results. 
\begin{theorem}\label{T1}
	Let $\Omega \subset \mathbb{R}^n$ be a bounded domain, $\delta > 0$ and $q >p-1$. Then, there exists an unbounded connected set $\Sigma \subset \mathbb{R} \times \left( W^{1,p}_{\mathrm{loc}}\cap C_0(\overline{\Omega})\right)$ of positive solutions of $(P)$ satisfying the following:
	\begin{itemize}
		\item[a)] $(0,0) \in \overline{\Sigma}$,
		\item[b)] $\Sigma$  contains the branch of minimal solutions of $(P)$ and bifurcates from infinity at $\lambda = 0$.
	\end{itemize}	 
	Moreover, by letting $\overline{\mbox{Proj}_{\mathbb{R}^+} \Sigma} = [0,\Lambda^*]$, then:
	\begin{itemize}
		\item[i)] $0<\Lambda^* \leq \lambda_1(\zeta+1)^{\delta}\zeta^{p-1}$, where $\zeta = \left(\frac{p-1+\delta}{q-p+1}\right)^{\frac{1}{q+\delta}}$,
		\item[ii)] for $\lambda > \Lambda^*$ there is no solution for $(P)$,
		\item[iii)] there exists  $0 < \Lambda_* \leq \Lambda^*$ such that for  $\lambda \in (0, \Lambda_*)$ there are at least two solutions of $(P)$ on $\Sigma$. In additional, $\Lambda_* = \Lambda^*$ if $\delta \in (0,1)$, that is, global multiplicity holds in this case.
		
	\end{itemize}
\end{theorem}
 In the proof of Theorem 1.1, we performed two approximation arguments, one in  the $(p-1)$-superlinear term $u^q$ truncating it by $\min\{n, u\}^{q-p+1}u^{p-1}$, and the other one in the singular nonlinearity $u^{-\delta}$ by considering $(u+\epsilon)^{-\delta}$ for $\epsilon>0$. The existence of an unbounded connected set of solutions for $(P)$ will be obtained  through a limit process of the continua $\Sigma_{\epsilon}^n$ as $n \rightarrow \infty$ and $\epsilon \rightarrow 0^+$, in that order, where $\Sigma^n_{\epsilon}$ are continua of positive solutions the $(\epsilon,n)$-problems
$$ (P_{\epsilon}^{n}) ~\left\{
\begin{array}{l}
	-\Delta_pu = {\lambda }\left[(u+\epsilon)^{-\delta} + \min\{n, u\}^{q-p+1}u^{p-1}\right] ~ \mbox{in } \Omega,\\
	u>0 ~ \mbox{in }\Omega,~~
	u=0  ~ \mbox{on }\partial\Omega.
\end{array}
\right.$$ 

The advantage of this approach is that several qualitative informations may be extracted from continua $\Sigma_{\epsilon}^n$ due to the linear and non-singular nature of the nonlinearity in $(P_{\epsilon}^{n})$. Moreover,  no additional conditions on the $(p-1)$-superlinear power $q$ are required.

The knowledge in literature (see \cite{Azorero} and \cite{Giaco}) and the behavior of the $\Sigma_\epsilon^n$ suggest that $(P)$ should admit a solution on the extremal parameter $\Lambda^*$, but we were not able to prove such existence in Theorem \ref{T1},  mainly by the fact that, under the assumptions considered there, any $\lambda \in (0, \Lambda^*)$ may be a bifurcation parameter of $\Sigma$ from  infinity.
This type of behavior may be ruled out when it is possible to assure the existence of a priori estimates for the solutions of $ (P) $ for each fixed $\lambda > 0$. 
 
 Inspired by  Azizieh and Cl\'ement \cite{Azizieh}, we proved a priori estimates
 by using the blow-up technique combined with Liouville type theorem for strictly convex and smooth domain. A crucial point in the proof presented by them is that the global maxima of the solutions are uniformly distanced from the  boundary $\partial \Omega$. In our case, this fact will be established through the results of monotonicity of  Damascelli and Sciunzi \cite{Dam}, but in this direction restrictions in the domain are essential.
 
\begin{theorem}\label{T2} Let $\Omega$ be a bounded strictly-convex and smooth domain in $\mathbb{R}^N$, $N \geq 2$. In additional, assume $ p-1 < q < p^* -1$. Then $\lambda = 0$ is the only bifurcation parameter of $\Sigma$ from infinity, that is, for $\{(\lambda_n, u_n)\}_{n=1}^\infty \subset \Sigma$ with 
	$$\left\{\begin{array}{l} \lambda_n \rightarrow  \lambda \\ \|u_n\|_{\infty} \rightarrow \infty\end{array}\right.$$ one has $\lambda = 0$. As a consequence, $(P)$ admits a solution for $\lambda = \Lambda^*$, that is,  global existence and local multiplicity hold.
\end{theorem}
About other bifurcation diagrams and  existence of extremal solutions with respect to the behaviour of the involved nonlinearity in the non-singular case, we  quote \cite{Cabre}, \cite{Dancer1}, \cite{Dancer2}, \cite{Dolb} and \cite{Miy} for instance.

The main novelties of this paper are the following:
\begin{itemize}
	\item[(i)] we propose here a different way to approach strongly-singular and $(p-1)$-superlinear problems, compared to variational aproaches used in most of papers in the current literature,
	\item[(ii)] the existence of a branch of positive solutions for $(P)$ is established without any restriction on the singular and $(p-1)$-superlinear powers,
	\item[(iii)] global multiplicity is proved for weak singularities, that is,  $\delta \in (0,1)$. This complements the principal result in \cite{Hirano} for  $q>p^* - 1$,
	\item[(iv)] local multiplicity is proved in the case of strong singularity ($\delta \in [1, \infty)$), which is completely new and complements the main results in \cite{Merida} and \cite{Bal},
	\item[(v)] the establishment of $\lambda = 0$ as the only bifurcation point of the continua from infinity in convex domains, according to our knowledge, is new.
\end{itemize}

 The outline of this paper is as follows. In section 2 we present several properties of the continuum $\Sigma_{\epsilon}^n$. Such properties are essential in section 3, where we will prove Theorems \ref{T1} and \ref{T2}.

\section{Approximated regular problem}

Throughout this paper, we will denote by $  e_p\in C_0^1(\overline{\Omega})$  the solution of the $p$-Laplacian torsion problem
\begin{equation}\label{torsion}
-\Delta_p u = 1 ~~ \mbox{in} ~\Omega, ~u|_{\partial \Omega} = 0
\end{equation}
and by $\omega_{\lambda,\epsilon} \in  C_0(\overline{\Omega})$ the only positive solution of
\begin{equation}\label{sp}
-\Delta_p u = \lambda(u+\epsilon)^{-\delta} ~~ \mbox{in} ~\Omega, ~u|_{\partial \Omega} = 0,
\end{equation}
for each $\lambda > 0$ and $\epsilon \geq 0$, where  $\omega_{\lambda,\epsilon} \in C^1_0(\overline{\Omega})$ for $\epsilon >0$ and $\omega_{\lambda,0} \in W_{\mathrm{loc}}^{1,p}(\Omega)\cap L^{\infty}(\Omega)$.

For each $n \in \mathbb{Z}^+$ and $\epsilon \geq  0$, let us consider the auxiliary problem
$$
(P_{\epsilon}^{n}) ~\left\{
\begin{array}{l}
-\Delta_pu = {\lambda }\left[(u+\epsilon)^{-\delta} + f_n(u)\right] ~ \mbox{in } \Omega,\\
u>0 ~ \mbox{in }\Omega,~~
u=0  ~ \mbox{on }\partial\Omega,
\end{array}
\right.
$$
where $f_n(u) = \min\{u,n\}^{q-p+1}u^{p-1}$.
 
\begin{lemma}\label{L1} Suppose  $\delta >0$.  Then, for each $\epsilon > 0$ and $n  \in \mathbb{Z}^+$,  there exists an unbounded \textit{continuum} (connected and closed) ${\Sigma}_{\epsilon}^n \subset \mathbb{R}^+ \times C_0(\overline{\Omega})$ of positive solutions of  $(P_{\epsilon}^n)$ emanating from $(0,0)$.
\end{lemma}
\begin{proof} It follows from the regularity results due to Lieberman (\cite{Lieberman}, Theorem 1) and Minty-Browder theorem (\cite{Dem}, Theorem 12.1) that for each $ (\lambda, v)  \in \mathbb{R}^+ \times C(\overline{\Omega})$, the problem
	\begin{equation}\label{2ii}
	-\Delta_pu = {\lambda }\left[(|v|+\epsilon)^{-\delta} + f_n(|v|)\right]   ~ \mbox{in } \Omega,  ~ ~   u=0  ~ \mbox{on } ~ \partial\Omega
	\end{equation}
	admits a unique solution  $u \in C^{1, \alpha}(\overline{\Omega}),$ for some $\alpha \in (0,1)$. Thus, the operator $T: \mathbb{R}^+ \times C(\overline{\Omega}) \rightarrow C(\overline{\Omega}),$ which associates  each pair $(\lambda, v) \in \mathbb{R}^+ \times C(\overline{\Omega})$ to the only weak solution of (\ref{2ii}), is well-defined.
	
	It is classical to show  that $T$ is a compact operator by using Arzel\`a-Ascoli  theorem. Hence, we are able to apply the bifurcation theorem by Rabinowitz (see Theorem 3.2 in \cite{MR0301587}) to get  an unbounded $\epsilon$-\textit{continuum} $\Sigma^n_{\epsilon} \subset \mathbb{R}^+ \times C_0(\overline{\Omega})$ of solutions of
	\begin{equation}\label{3}
	T(\lambda, u) = u.
	\end{equation}
	By the definition, $T(0, v) = 0$. Moreover, if $T(\lambda, 0) = 0$, then $\lambda = 0$. So we  conclude that $\Sigma^n_{\epsilon}\backslash \{(0,0)\}$  is formed by nontrivial solutions of (\ref{3}).
	
	Finally, using that $0<(|v|+\epsilon)^{-\delta} + f_n(|v|) \in L^{\infty} (\Omega)$ for each given $v \in C(\overline{\Omega})$ and classical strong maximum principle from Vazquez \cite{vazquez}, we obtain that $T((\mathbb{R}^+\backslash \{0\}) \times C(\overline{\Omega})) \subset  \{ u \in C_0(\overline{\Omega}) : u > 0 ~\mbox{in} ~ \Omega\}$. Therefore, $\Sigma^n_{\epsilon}$ is a \textit{continuum} of positive solutions of $(P^n_{\epsilon})$, for any $\epsilon>0$ and $n \in \mathbb{Z}^+$ given. 
\end{proof}
\fim
\vspace{0.4cm}

\begin{lemma}\label{L5}
	Let $\epsilon > 0$ and $n \in \mathbb{N}$. Then:
\begin{itemize}
	\item[a)] there exists $\Lambda_{\epsilon,n} = \Lambda_{\epsilon,n}(\epsilon, n, \Omega, \lambda_1, p, q)$ such that $(P_{\epsilon}^n)$ has no solution for $\lambda > \Lambda_{\epsilon,n}$.
	\item[b)] the value $\lambda = \lambda_1/n^{q-p+1}$ is the unique bifurcation point of $\Sigma_{\epsilon}^n$ from infinity.
\end{itemize} 
\end{lemma}
\begin{proof} 
	\begin{itemize}
\item[a)]	Since $\left[(t+\epsilon)^{-\delta} + f_n(t)\right]t^{1-p} \rightarrow n^{q-p+1}$ as $t \rightarrow +\infty$, we can find $C_* > 0$ such that 
\begin{equation}\label{contr}
(t+\epsilon)^{-\delta} + f_n(t)\geq C_* t^{p-1}, ~~\mbox{for all} ~  t \geq 0.
\end{equation}
We claim that there is no solution of $(P_{\lambda, \epsilon}^n)$ for $\lambda > \lambda_1C_*^{-1}$. Indeed, if $u_*>0 $ solves $(P_{\lambda, \epsilon}^n)$ for some $\lambda > \lambda_1C_*^{-1}$, then $u_*$ is a supersolution of 
\begin{equation}\label{autovalor} \left\{\begin{array}{ll}
-\Delta_p u = (\lambda_1 + \kappa)u^{p-1} ~~\mbox{in} ~\Omega,\\
u > 0, ~~~u|_{\partial \Omega} = 0,
\end{array}\right.
\end{equation}
for all $\kappa \in \left(0, \lambda C_* - \lambda_1\right)$. On the other hand, $s\phi$ is a subsolution of (\ref{autovalor}) satisfying $s\phi_1 < u_*$ in $\Omega$ for $s \in (0,1)$ small enough. So, the monotonic interation method provides a solution of (\ref{autovalor}) for any $\kappa \in \left(0, \lambda C_* - \lambda_1\right)$, which contradicts the fact that $\lambda_1$ is an isolate point in the spectrum of $(-\Delta_p, W_0^{1,p}(\Omega))$ (see \cite{ANANE}).

 \item[b)] Since $\Sigma_{\epsilon}^n$ is bounded in the $\lambda-$direction, it has to become unbounded in the direction of  the Banach space $C_0(\overline{\Omega})$. Hence, there exists $\lambda_* > 0$ and a sequence $\{(\lambda_k, u_k)\}_{k=1}^{\infty} \subset \Sigma_{\epsilon}^n$ such that 
 $$\left\{\begin{array}{l}
 \|u_k\|_{\infty} \rightarrow +\infty \\
 \lambda_k \rightarrow \lambda_*
 \end{array} \right.$$
Then $v_k := u_k/\|u_k\|_{\infty}$ satisfies
\begin{equation}\label{hththt}\left\{\begin{array}{l}-\Delta_p v_k = \lambda_k \left[n^{q-p+1}v_k^{p-1} + \frac{g(u_k)}{\|u_k\|_{\infty}^{p-1}}\right], ~~\mbox{in} ~\Omega\\
v_k > 0, ~~\|v_k\|_\infty = 1, ~~v_k|_{\partial \Omega} = 0,
\end{array}\right.
\end{equation}
where $$g(t) = (t+\epsilon)^{-\delta} + f_n(t) - n^{q-p+1}t^{p-1} = \left\{\begin{array}{l} 
(t+\epsilon)^{-\delta} + t^q - n^{q-p+1}t^{p-1}, ~~\mbox{if} ~ 0 \leq t \leq n \\
(t+\epsilon)^{-\delta}, ~~\mbox{if} ~ n < t
\end{array}\right.
$$

Notice that $|g(t)|\leq \epsilon^{-\delta} + 2n^q, $ for all $t \in \mathbb{R}^+,$ so that the right hand side of the equation in (\ref{hththt}) is uniformly bounded in $L^{\infty}(\Omega)$. Hence, we can employ the regularity result due to Lieberman (\cite{Lieberman}, Theorem 1) to conclude that 
$\{v_k\}_{k=1}^{\infty}$ is uniformly bounded in $C^{1,\beta}(\overline{\Omega})$, for some $\beta \in (0,1)$. Therefore, it follows from Arzel�-Ascoli theorem that exists $ v \in C^1(\overline{\Omega})$ and a subsequence $\{v_{k_j}\}_{j=1}^{\infty}$ such that $v_{k_j} \rightarrow v$ in $C^1(\overline{\Omega})$ as $k_j \rightarrow \infty$. Moreover, $ v \geq 0$ in $\Omega$ , $\|v\|_{\infty} =1$ and $v$ solves 
$$-\Delta_p v = \lambda_*n^{q-p+1}v^{p-1} ~\mbox{in} ~\Omega, ~~~v|_{\partial \Omega} = 0. $$ Since $v$ does not change signs, it follows  from (Theorem 7, \cite{Lid})  that $\lambda_* = \lambda_1/n^{q-p+1}$. This concludes the proof of Lemma \ref{L5}.
\end{itemize}
\end{proof}
\fim

The next two lemmas claim that the continuum $\Sigma_{\epsilon}^n$ cross the line $\lambda = \lambda_1/n^{q-p+1}$.
\begin{lemma}\label{L6} Let
$$
\varepsilon_0 = 	2^{\frac{-q}{(q+\delta)(p-1)}-1}, ~~~~  N_0 = \lambda_1^{\frac{1}{q-p+1}}\left(\|\omega_{1,0}\|_{\infty} \varepsilon_0^{-1}\right)^{\frac{p-1+\delta}{q-p+1}} +1, 
$$
where $\omega_{1,0}$ is the only solution of (\ref{sp}),  with $\lambda=1$ and $\epsilon=0$.
Then for each $0 < \epsilon < \varepsilon_0 $ and $n> N_0$,  there exists $ \varrho > 0$, independent of $\epsilon>0$ and $n$, such that $(P_{\epsilon}^n)$ admits a solution for all $\lambda \in \left[ \frac{\lambda_1}{n^{q-p+1}}, \frac{\lambda_1}{n^{q-p+1}} + \varrho\right)$.
\end{lemma}
\begin{proof}
	Clearly the only solution $\omega_{ \lambda,\epsilon} \in W_0^{1,p}(\Omega)\cap C_0(\overline{\Omega})$ of the problem (\ref{sp}) is a subsolution of $(P_{\epsilon}^n)$. To construct a supersolution of $(P_{\epsilon}^n)$, let us consider $\overline{\omega}_{\lambda, \epsilon} = M\omega_{\lambda, \epsilon}$, where $M > 1$ will be chosen later. In order to $\overline{\omega}_{\lambda, \epsilon}$ be a supersolution of $(P_{\epsilon}^n)$, we must have
	\begin{eqnarray} \label{21}
		-\Delta_p \overline{\omega}_{\lambda,\epsilon} = {M^{p-1}\lambda}(\omega_{ \lambda,\epsilon,} +  \epsilon)^{-\delta} \geq {\lambda}\left[(M\omega_{\lambda,\epsilon,} + \epsilon)^{-\delta} + f_n(M\omega_{\lambda,\epsilon})\right] ~~\mbox{in} ~ \Omega.
	\end{eqnarray}
	
For the validity of the inequality (\ref{21}), it is sufficient that
\begin{equation}\label{zzz}\left\{\begin{array}{ccc}
\displaystyle\frac{M^{p-1}}{2(\omega_{\lambda,\epsilon} + \epsilon)^{\delta}} &\geq& \displaystyle\frac{1}{(\omega_{\lambda,\epsilon} + \epsilon)^{\delta}} \\
\displaystyle\frac{M^{p-1}}{2(\omega_{\lambda,\epsilon} + \epsilon)^{\delta}} &\geq& M^q(\omega_{\lambda, \epsilon} + \epsilon)^q 
\end{array}\right.
\end{equation}
hold.

Since $\omega_{\lambda, \epsilon}$ is a subsolution of (\ref{sp})
and $\omega_{\lambda, 0} = \lambda^{\frac{1}{p-1+\delta}}\omega_{1,0}$ is the only solution of (\ref{sp}) with $\epsilon = 0$, we can apply the comparison principle of \cite{Can, ZAMP} to conclude that
\begin{equation}\label{tststs}
\omega_{\lambda, \epsilon} \leq \lambda^{\frac{1}{p-1+\delta}}\omega_{1,0} ~\mbox{in} ~~\Omega.
\end{equation}

On the other hand, by the choice of $N_0$ and $\varepsilon_0$, we have
$${\left(\frac{\lambda_1}{n^{q-p+1}}\right)^{\frac{1}{p-1+\delta}}\|\omega_{1,0}\|_\infty + \epsilon < 2^{\frac{-q}{(q+\delta)(p-1)}} }$$
for all $0 < \epsilon < \varepsilon_0$ and $n > N_0$. Hence, for any $0 < \varrho < \left(\frac{\varepsilon_0}{\|\omega_{1,0}\|_\infty}\right)^{p-1+\delta} - \frac{\lambda_1}{N_0^{q-p+1}}$ and for all  $\lambda \in \left[\lambda_1/n^{q-p+1}, \lambda_1/n^{q-p+1} +\varrho\right)$
we obtain
\begin{equation}\label{ysysysys}
{ \lambda^{\frac{1}{p-1+\delta}}\|\omega_{1,0}\|_\infty + \epsilon < 2^{\frac{-q}{(q+\delta)(p-1)}},} 
\end{equation}
for all  $0 < \epsilon < \varepsilon_0$ and $n > N_0$.
 Thus, combining (\ref{tststs}) and (\ref{ysysysys}) one gets  
 \begin{equation}\label{lll}
 \|\omega_{\epsilon, \lambda}\|_{\infty} + \epsilon <2^{\frac{-q}{(q+\delta)(p-1)}},
 \end{equation}
 for all  $\lambda \in \left[\lambda_1/n^{q-p+1}, \lambda_1/n^{q-p+1} +\varrho\right)$, $\epsilon \in (0, \varepsilon_0) $ and $n > N_0$. So, by setting $M=2^{\frac{1}{p-1}}$, it follows from (\ref{lll}) that
 \begin{equation}\label{subsup}
 M<  \left(\frac{1}{2(\|\omega_{\epsilon, \lambda}\|_{\infty} + \epsilon)^{\delta + q}}\right)^{\frac{1}{q-p+1}}.
 \end{equation}
As a consequence of (\ref{subsup}), both inequalities in (\ref{zzz}) are fulfilled for such $M$, whence  $\overline{\omega}_{\epsilon, \lambda} = M{\omega}_{\epsilon, \lambda}$ is a supersolution of $(P_{\epsilon}^n)$. 
 
 Now, let us prove that for $\lambda \in \left[\lambda_1/n^{q-p+1}, \lambda_1/n^{q-p+1} +\varrho\right)$, $0 < \epsilon < \varepsilon_0$ and $n> N_0$, the problem $(P_{\epsilon}^n)$ admits a solution in $[\omega_{\lambda, \epsilon}, \overline{\omega}_{\lambda, \epsilon}]$. For this, consider
 $$g(x,t) = \left\{\begin{array}{lll}
 (\omega_{\lambda, \epsilon} + \epsilon)^{-\delta} + f_n(\omega_{\lambda, \epsilon}) &\mbox{if} &t < \omega_{\lambda, \epsilon}, \\
 (t + \epsilon)^{-\delta} + f_n(t) &\mbox{if} & \omega_{\lambda, \epsilon} \leq t \leq \overline{\omega}_{\lambda, \epsilon}, \\
 (\overline{\omega}_{\lambda, \epsilon} + \epsilon)^{-\delta} + f_n(\overline{\omega}_{\lambda, \epsilon}) &\mbox{if} & \overline{\omega}_{\lambda, \epsilon} < t
 \end{array}\right.
$$
and the functional $I : W_0^{1,p}(\Omega) \rightarrow \mathbb{R}$ given by
$$ I(u) = \frac{1}{p}\displaystyle\int_\Omega |\nabla u|^pdx - \lambda\displaystyle\int_\Omega G(x,u)dx, $$ 
where $G(x,t) = \int_0^t g(x,s)ds$. As usual, we can prove that $I \in C^1(W_0^{1,p}, \mathbb{R})$ is coercive and weakly lower semi-continuous. Hence,
$I$ achieves its global minimum at some $u_0 \in  W_0^{1,p}(\Omega)$, which satisfies 
$-\Delta_p u_0 = \lambda g(x,u_0)$ in $\Omega$.

Moreover, 
\begin{eqnarray*}
0 &\leq& \displaystyle\int_{\Omega} \left(|\nabla \omega_{\lambda, \epsilon}|^{p-2}\nabla \omega_{\lambda, \epsilon}- |\nabla u_0|^{p-2}\nabla u_0\right)\nabla (\omega_{\lambda, \epsilon} - u_0)^+dx \\
& \leq & \lambda\displaystyle\int_{[\omega_{\lambda, \epsilon} > u_0]} \left((\omega_{\lambda, \epsilon}+\epsilon)^{-\delta} + f_n(\omega_{\lambda, \epsilon}) - g(x,u_0)\right)(\omega_{\lambda, \epsilon} - u_0)^+dx = 0,
\end{eqnarray*}
whence $u_0 \geq \omega_{\lambda, \epsilon}$ in $\Omega$. In a similar way, we obtain $u_0 \leq \overline{\omega}_{\lambda, \epsilon}$ in $\Omega$. Therefore, $u_0$ is the solution claimed.
\end{proof}
\fim

\begin{lemma}\label{L7}
For all $0 \leq  \epsilon <  \frac{1}{2}\left(\frac{p-1+\delta}{q-p+1}\right)^{\frac{1}{q+\delta}} < n$ and for each $\lambda \in \displaystyle\mbox{Proj}_{\mathbb{R}} \Sigma_{\epsilon}^n$,  the problem $(P_{\epsilon}^n)$ admits at most a solution satisfying $\|u\|_\infty \leq \frac{1}{2}\left(\frac{p-1+\delta}{q-p+1}\right)^{\frac{1}{q+\delta}}.$
\end{lemma}
\begin{proof} Define $$g_{\epsilon}(t) = \frac{(t+\epsilon)^{-\delta} + f_n(t)}{t^{p-1}}, ~~~\mbox{for} ~~ t > 0. $$
If $0 < t \leq n $, then $g_{\epsilon}$ becomes $g_{\epsilon}(t) = t^{1-p}(t+\epsilon)^{-\delta} + t^{q-p+1}$, whose derivative is
\begin{eqnarray*}
g'_{\epsilon}(t) &=& 
(1-p)t^{-p}(t+\epsilon)^{-\delta} -\delta t^{1-p}(t+\epsilon)^{-\delta -1} + (q-p+1)t^{q-p} \\
& = &t^{-p}(t+\epsilon)^{-\delta -1} \left[(1-p)(t+\epsilon) - \delta t + (q-p+1)t^{q}(t+\epsilon)^{1+\delta}\right] \\
&=& t^{-p}(t+\epsilon)^{-\delta -1} \left[t(1-p-\delta)) + \epsilon(1-p) + (q-p+1)t^{q}(t+\epsilon)^{1+\delta}\right] \\
&< & t^{-p}(t+\epsilon)^{-\delta -1}\left[t(1-p-\delta))  +(q-p+1)t^{q}(t+\epsilon)^{1+\delta} \right]. \end{eqnarray*}
Thus, for all $0 < t \leq n$ satisfying 
\begin{equation}\label{228}
t+\epsilon \leq \left(\frac{p-1+\delta}{q-p+1}\right)^{\frac{1}{q+\delta}},
\end{equation}
we have
$$
(q-p+1)t^{q-1}(t+\epsilon)^{1+\delta}  < (q-p+1)(t+\epsilon)^{q+\delta} <  (p-1+\delta), 
$$
which implies in $g_{\epsilon}'(t) < 0$. In particular, if
$$0 \leq  \epsilon <  \frac{1}{2}\left(\frac{p-1+\delta}{q-p+1}\right)^{\frac{1}{q+\delta}} ~~~\mbox{and} ~~~~0 < t <  \frac{1}{2}\left(\frac{p-1+\delta}{q-p+1}\right)^{\frac{1}{q+\delta}},$$ then (\ref{228}) holds and as a result
 $g_{\epsilon}$ is decreasing. Hence, we can employ the uniqueness result of D\'iaz-Saa (see Theorem 1 in \cite{DIAZ}) to conclude that for each $\lambda \in \displaystyle\mbox{Proj}_{\mathbb{R}} \Sigma_{\epsilon}^n$, the problem $(P_{\epsilon}^n)$ has a unique solution satisfying $\|u\|_{\infty} \leq \frac{1}{2}\left(\frac{p-1+\delta}{q-p+1}\right)^{\frac{1}{q+\delta}}$. 

\end{proof}
\fim

Combining the previous lemmas, we have the following.
\begin{lemma}\label{L8}
There exist $\varepsilon_1 > 0$ (independent of $n>0$) and $N_1 > 0$ (independent of $\epsilon>0$) such that $ \left(\displaystyle\mbox{Proj}_{\mathbb{R}} \Sigma_{\epsilon}^n \displaystyle\cap ({\lambda_1}/{n^{q-p+1}}, +\infty) \right) \times C_0(\overline{\Omega})_+ \neq \emptyset$, for all $0 < \epsilon < \varepsilon_1$ and $n > N_1$.
\end{lemma}
\begin{proof}
Let $\varepsilon_0$, $N_0$, $K$ and $\varrho$ be the constantes introduced in Lemma \ref{L6}. In Lemma \ref{L6} we proved that $(P_{\epsilon}^n)$ admits a positive solution for all 
$\lambda \in \left[ \frac{\lambda_1}{n^{q-p+1}}, \frac{\lambda_1}{n^{q-p+1}} + \varrho\right)$, whenever $0 < \epsilon < \varepsilon_0$ and $n> N_0$. Moreover, for $\lambda \in \left[ \frac{\lambda_1}{n^{q-p+1}}, \frac{\lambda_1}{n^{q-p+1}} + \varrho\right)$ the solution obtained there, say u, satisfies
\begin{equation}\label{eq116}\|u\|_\infty \leq 2^{\frac{1}{p-1}}\lambda^{\frac{1}{p-1+\delta}}\|\omega_{1,0}\|_{\infty}. \end{equation}
Thus, by taking $\varepsilon_1 = \min\left\{\varepsilon_0, \frac{1}{2}\left(\frac{p-1+\delta}{q-p+1}\right)^{\frac{1}{q+\delta}} \right\}$,  $$ N_1 = \max\left\{N_0, \frac{1}{2}\left(\frac{p-1+\delta}{q-p+1}\right)^{\frac{1}{q-p}}, \left(2^{\frac{p}{p-1}}\left(\frac{q-p+1}{p-1+\delta}\right)^{\frac{1}{q+\delta}}\lambda_1^{\frac{1}{p-1+\delta}}\|\omega_{1,0}\|_{\infty}\right)^{\frac{p-1+\delta}{q-p+1}} \right\}, $$
	and reducing $\varrho$, if it is necessary, we conclude  from (\ref{eq116}) that $\|u\|_\infty < \frac{1}{2}\left(\frac{p-1+\delta}{q-p+1}\right)^{\frac{1}{q+\delta}}$,  whenever $0 < \epsilon < \varepsilon_1$, $n > N_1$ and  $\lambda \in \left[ \frac{\lambda_1}{n^{q-p+1}}, \frac{\lambda_1}{n^{q-p+1}} + {\varrho}\right).$ Indeed, for $n > N_1$, $\epsilon \in (0, \varepsilon_1)$ and for all $\lambda \in \left[ \frac{\lambda_1}{n^{q-p+1}}, \frac{\lambda_1}{n^{q-p+1}} + \varrho\right)$ one has
	$$\lambda^{\frac{1}{p-1+\delta}}\|\omega_{1,0}\|_\infty + \epsilon < 2^{\frac{-q}{(q+\delta)(p-1)}} $$
	and 
	$$\|u\|_\infty+\epsilon \leq \left(\frac{p-1+\delta}{q-p+1}\right)^{\frac{1}{q+\delta}}.$$		
	 After all these, the result follows by combining Lemmas \ref{L6} and \ref{L7}.
\end{proof}
\fim
\vspace{0.3cm}

 The previous lemmas  suggest that $\Sigma_{\epsilon}^n$ may have one of the shapes given in the  figure below.

\begin{figure}[h!]
	\centering
	\begin{tikzpicture}[scale=0.70]
	
\draw[thick, ->] (-14, 0) -- (-5, 0);
	\draw[thick, ->] (-13, -1) -- (-13, 8);
	\draw[very thick] (-13.0 ,0.0) .. controls (-6.0,1.0) and (-8.0,3.0) .. (-9.0, 4.0);
	\draw[very thick] (-9.0 ,4) .. controls (-9.55, 5.0) and (-9,5.5) .. (-8.88, 8.0);
	\draw[very thick, dotted] (-8.83,0) -- (-8.83, 8.0);
	\draw (-5,0) node[below]{$\lambda$};
	\draw (-8.83,0) node[below]{$\frac{\lambda_1}{n^{q-p+1}}$};
	\draw (-13, 0) node[below left]{$(0,0)$};
	\draw (-13,8) node[left]{$\|u\|_\infty$};
	\draw (-7.83,3) node[right]{{\Large $\Sigma_{\epsilon}^n$}};
	\draw[thick, ->] (0, 0) -- (9, 0);
	\draw[thick, ->] (1, -1) -- (1, 8);
	\draw[very thick] (1.0 ,0.0) .. controls (8.0,1.0) and (7.0,3.0) .. (5.8, 4.0);
	\draw[very thick] (5.8 ,4) .. controls (5.1, 4.9) and (5.25,5.6) .. (5.24, 8.0);
	\draw[very thick, dotted] (5.17,0) -- (5.17, 8.0);
	\draw (9,0) node[below]{$\lambda$};
	\draw (5.17,0) node[below]{$\frac{\lambda_1}{n^{q-p+1}}$};
	\draw (1, 0) node[below left]{$(0,0)$};
	\draw (1,8) node[left]{$\|u\|_\infty$};
	\draw (7,3) node[right]{{\Large $\Sigma_{\epsilon}^n$}};
	\draw (-7.5,-2) node[right]{Fig.1. Possible bifurcation diagrams of $\Sigma_{\epsilon}^n$};
	\end{tikzpicture}
\end{figure}

\section{Asymptotic singular problem}
The unbounded connected set of solutions
 of $(P)$ will be obtained through limit process of $\Sigma_{\epsilon}^n$ as $n \rightarrow +\infty$ and $\epsilon \rightarrow 0^+$. 
 \begin{definition}\label{D3}	Let $X$ be a Banach space and let $\{\Sigma_n \}_{n=1}^{\infty}$ be a family of subsets of $X$. The set of all
 	points $x \in X$  such that every neighborhood of $x$ contains points of infinitely many
 	sets $\{\Sigma_n \}_{n=1}^{\infty}$ is called the limit superior of $\{\Sigma_n \}_{n=1}^{\infty}$ and is written $\displaystyle\lim_{n \rightarrow \infty}\sup \Sigma_n$. The set
 	of all points $y$ such that every neighborhood of $y$ contains points of all but a finite number of
 	the sets of $\{\Sigma_n \}_{n=1}^{\infty}$ is called the limit inferior of $\{\Sigma_n \}_{n=1}^{\infty}$ and is written $\displaystyle\lim_{n \rightarrow \infty}\inf \Sigma_n$.
 \end{definition}
\begin{lemma}\label{L9}(\cite{DAI})
	Let $X$ be a normal space and let $\{\Sigma_n \}_{n=1}^{\infty}$ be a sequence of unbounded connected
	subsets of $X$. Assume that:
	\begin{itemize}
		\item[i)] there exists $z^* \in \displaystyle\lim_{n \rightarrow \infty}\inf \Sigma_n$ with $\|z^*\| < +\infty$, 
		\item[ii)] for every $R>0$, $\left(\displaystyle\bigcup_{n=1}^{+\infty}\Sigma_n\right) \cap \overline{B_R(0)}$ is a relatively compact set of $X$, where
		$\overline{B_R(0)}  = \{x \in X ~: ~\|x\| \leq R\}.$
		\end{itemize}
		Then $C = \displaystyle\lim_{n \rightarrow \infty}\sup \Sigma_n$ is unbounded, closed and connected set.
\end{lemma}
\begin{lemma}\label{L11} (\cite{DAI})
	Let $(X, \|\cdot \|)$ be a normed vector space and $\{\Sigma_n \}_{n=1}^{\infty}$ a sequence of unbounded sets whose limit superior is $C$ and satisfies the following conditions:
	\begin{itemize}
		\item[i)] there exists $z^* \in C $ with $\|z^*\| <+\infty$,
		\item[ii)] $\left(\displaystyle\bigcup_{n=1}^\infty \Sigma_n \right)\cap \overline{B_R(z^*)}$	
		is a relatively compact, for every $R> 0$.
	\end{itemize}
	Then, for each $\epsilon > 0$ there exists an $m \in \mathbb{N}$ such that  $\Sigma_n \subset V_{\epsilon}(C)$ for all $n > m$,
	where $V_{\epsilon}(C) = \{y \in X ~: ~\mbox{dist}(y,C) < \epsilon\}$.
	
\end{lemma}

\begin{proposition}\label{P1}
	For each $0 < \epsilon < \varepsilon_1$, the problem	
	$$
(P_{\epsilon})~	\left\{
	\begin{array}{l}
	-\Delta_pu = {\lambda }\left[(u+\epsilon)^{-\delta} + u^{q}\right]  ~ \mbox{in } \Omega,\\
	u>0 ~ \mbox{in }\Omega,~~
	u=0  ~ \mbox{on }\partial\Omega
	\end{array}
	\right.
$$
	admits a continuum $\Sigma_{\epsilon} \subset W_0^{1,p}(\Omega)\cap C_0(\overline{\Omega})$ of positive solutions, which is bounded in the $\lambda$-direction, emanates from $(0,0)$ and bifurcates from infinity at $\lambda =0$.  
	\end{proposition}
\begin{proof}
We will apply Lemma \ref{L9} to get such continuum. More precisely, 	consider $X = \mathbb{R}\times C_0(\overline{\Omega})$,  $0 < \epsilon < \varepsilon_1$ and  $n > N_1$. For any $n \in \mathbb{N}$, the continuum $\Sigma_{\epsilon}^{n}$ contains $(0,0)$, whence $(0,0) \in  \displaystyle\lim_{n \rightarrow \infty}\inf \Sigma_{\epsilon}^{n} $. Moreover, by taking $$\{(\lambda_j, u_j)\}_{j=1}^{\infty} \subset \left(\displaystyle\bigcup_{n=N_1+1}^{+\infty}\Sigma_{\epsilon}^{n}\right) \cap \overline{B_R(0,0)}, $$ it follows from the mapping properties of the inverse p-Laplacian (see \cite{Lieberman}) that $\|u_j\|_{C^{1,\beta}(\overline{\Omega})}$ is uniformly bounded, for some $\beta \in (0,1)$. Thus, Arzel\`a-Ascoli theorem assures us that  
$$
\lambda_j \rightarrow \lambda \geq 0 ~~\mbox{and} ~~
u_j \rightarrow u ~\mbox{in} ~C^{1}(\overline{\Omega}),
$$
up to a subsequence.
Hence, we are able to apply Lemma \ref{L9} to conclude that $\Sigma_{\epsilon} := \displaystyle\lim_{n \rightarrow \infty}\sup \Sigma^{n}_\epsilon$ is unbounded, closed and connected set in $\mathbb{R}\times C_0(\overline{\Omega})$. 

We claim that $\Sigma_{\epsilon}$ is formed by solutions of $(P_{\epsilon})$. In fact, if $(\lambda, u) \in \Sigma_{\epsilon}$, then  
$$(\lambda_{n_j}, u_{n_j})   \rightarrow (\lambda, u) ~~\mbox{in} ~ \mathbb{R}\times C_0(\overline{\Omega}), $$  for  some subsequence $\{(\lambda_{n_j}, u_{n_j})\},$ where $(\lambda_{n_j}, u_{n_j})\in \Sigma^{n_j}_{\epsilon}$.
In particular, $\{\|u_{n_j}\|_\infty\}_{j=1}^{\infty}$ is uniformly bounded, thus once again invoking Lieberman regularity result \cite{Lieberman} and applying Arzel\`a-Ascoli theorem we obtain that $(\lambda_{n_j}, u_{n_j})   \rightarrow (\lambda, u) ~~\mbox{in} ~ \mathbb{R}\times C^1_0(\overline{\Omega})$ and $(\lambda, u)$ solves $(P_\epsilon)$.

To prove that $\Sigma_\epsilon$ is bounded in the $\lambda$-direction, notice that the function
$$g_{\epsilon}(t) = \frac{(t+\epsilon)^{-\delta} + t^{q}}{t^{p-1}}, ~~~\mbox{for} ~~ t > 0 $$ admits a global minimum at $t_{\mathrm{min}} = h^{-1}\left(0\right),$ where $h:\mathbb{R}^+ \rightarrow \mathbb{R}$ is an invertible function given by $$h(t) = -\delta(t+\epsilon)^{-\delta-1}t^{1-p} + (1-p)(t+\epsilon)^{-\delta}t^{-p} + (q-p+1)t^{q-p}.$$ Moreover, 
\begin{equation}\label{tmin}
\left(\frac{p-1+\delta}{q-p+1}\right)^{\frac{1}{q+\delta}} - \epsilon < t_{\mathrm{min}} < \left(\frac{p-1+\delta}{q-p+1}\right)^{\frac{1}{q+\delta}}.
\end{equation}
So, denoting by $\zeta = \left(\frac{p-1+\delta}{q-p+1}\right)^{\frac{1}{q+\delta}}$, we conclude from (\ref{tmin}) that 
$$g_\epsilon(t) \geq g_\epsilon(t_{\mathrm{min}}) \geq (\zeta+1)^{-\delta}\zeta^{1-p} $$
for all $t > 0$ and $0 < \epsilon < \varepsilon_1 < 1$.

Suppose there exists $({\lambda_*}, {u_*}) \in \Sigma_{\epsilon}$ with ${\lambda_*} >  \lambda_1(\zeta+1)^{\delta}\zeta^{p-1}$. Then, ${\lambda_*} g_{\epsilon}(t) \geq \lambda_1 + \kappa$ for every $\kappa > 0$  small enough, that is, $${\lambda_*}\left((t+\epsilon)^{-\delta} + t^{q}\right) \geq (\lambda_1 + \kappa){t^{p-1}}, ~~~\mbox{for all} ~~ t > 0.$$ In particular, ${u_*}$ is a supersolution of 
\begin{equation}\label{aut}
\left\{\begin{array}{l}
-\Delta_p u = (\lambda_1 + \kappa){u^{p-1}} ~~\mbox{in}~~\Omega, \\
u > 0, ~~ u|_{\partial \Omega} = 0,
\end{array}\right.
\end{equation}
for all $\kappa > 0$ small enough. 
Moreover, for $s > 0$ small  $s\phi_1$ is a subsolution of 
 (\ref{aut}) and satisfies ${u_*} \geq s\phi_1$ in $\Omega$. Hence, by monotone interaction we obtain a solution of (\ref{aut}) for any $\kappa > 0$ small, contradicting the fact that $\lambda_1$ is an isolated point in the spectrum of $(-\Delta_p, W_0^{1,p}(\Omega))$ (see \cite{ANANE}). Therefore,   $\mbox{Proj}_{\mathbb{R}^+} \Sigma_{\epsilon} \subset [0,  \lambda_1(\zeta+1)^{\delta}\zeta^{p-1}],$ for any $0 < \epsilon < \varepsilon_1$.

 Finally, let us prove that $\Sigma_{\epsilon}$ joins $(0,0)$ to $(0, +\infty)$. In this direction, we first observe that there exists a sequence $\{(\lambda_k, u_k)\}_{k=1}^{\infty} \subset \Sigma_\epsilon$ such that $\lambda_k \rightarrow 0^+$ and $u_k \neq \underline{u}_{\lambda_k}$, where  $\underline{u}_{\lambda_k}$ denotes the minimal solution of $(P_{\epsilon})$ for $\lambda = \lambda_k$ . Indeed, otherwise we could find some $\lambda_* > 0$ small enough such that $\Sigma_\epsilon \cap [0,\lambda_*] \times C_0(\overline{\Omega}) $ contains only elements in the branch of minimal solution of $(P_\epsilon)$ (see Proposition \ref{P2}-i) below), which is not possible by invoking Lemma \ref{L11}. Therefore, consider a sequence $ \{(\lambda_k, u_k)\}_{k=1}^{\infty} \subset \Sigma_\epsilon$ satisfying $\lambda_k \rightarrow 0^+$ and $u_k \neq \underline{u}_{\lambda_k}$. In this case, we must have $\|u_k\|_{\infty} \rightarrow \infty$, up to a subsequence. On the contrary, $\|u_k\|_{\infty}$ would be uniformly bounded, which combined with the $\lambda_k \rightarrow 0^+$ and Arzel�-Ascoli theorem would lead us to $(\lambda_k, u_k) \rightarrow (0,0)$ in $\mathbb{R}\times C_0(\overline{\Omega})$, but  this is not possible by uniqueness of solution for small  $\lambda$ and small norm (note that $t \mapsto [(t+\epsilon)^{-\delta} + t^q]/t^{p-1}$ is decreasing for $0 < t < \eta$, $\eta$ small). Hence the continuum $\Sigma_{\epsilon}$ joins $(0,0)$ to $(0, +\infty)$.

\end{proof}
\fim

\begin{proposition}\label{P2}
	For each $0< \epsilon < \varepsilon_1$, let ${\overline{\mbox{Proj}_{\mathbb{R}^+} \Sigma_{\epsilon}} = [0, \Lambda_\epsilon]}$ be the closure of the projection of ${\Sigma}_{\epsilon}$ onto the $\lambda-$axis. Then:
	\begin{itemize}
		\item[i)] $\Sigma_\epsilon$ contains the branch of minimal solutions of $(P_\epsilon)$,
		\item[ii)] for $\lambda > \Lambda_\epsilon$ there is no solution of $(P_{\epsilon})$,
		\item[iii)] for $0 < \lambda < \Lambda_\epsilon$ there are at least two solutions of $(P_{\epsilon})$ on $\Sigma_\epsilon$,
		\item[iv)] the map $\epsilon \mapsto \Lambda_{\epsilon}$ is non decreasing.
		\end{itemize}
		Besides this, $0<\Lambda_\epsilon \leq \lambda_1(\zeta+1)^{\delta}\zeta^{p-1}$ for all $\epsilon>0$ sufficiently small, where $\zeta = \left(\frac{p-1+\delta}{q-p+1}\right)^{\frac{1}{q+\delta}}$.	
\end{proposition}
\begin{proof} \textit{Part $i)$:}	Since $\Sigma_\epsilon \subset \mathbb{R}^+\times C_0(\overline{\Omega})_+$, it follows from the theory of regularity for elliptic equations (see \cite{Lieberman}, Theorem 1) that $\Sigma_\epsilon \subset \mathbb{R^+}\times C^1_0(\overline{\Omega})_+$. Let us denote by $(\Sigma_\epsilon, \mathbb{R}\times C)$  the set $\Sigma_\epsilon$ with the topology induced by $\mathbb{R}\times C_0(\overline{\Omega})$ and represent by $(\Sigma_\epsilon, \mathbb{R}\times C^1)$  the set $\Sigma_\epsilon$ with the topology induced by $\mathbb{R}\times C_0^1(\overline{\Omega})$. As we have proved, $(\Sigma_\epsilon, \mathbb{R}\times C)$ is connected.
	
\noindent	\textit{Claim:} $(\Sigma_\epsilon, \mathbb{R}\times C^1)$ is connected. Indeed, let $\mathbb{Z}$ be the set of integers with the topology induced by the usual topology on $\mathbb{R}$ and $h:  (\Sigma_\epsilon, \mathbb{R}\times C^1) \rightarrow \mathbb{Z}$ be a continuous function. Then $h:  (\Sigma_\epsilon, \mathbb{R}\times C) \rightarrow \mathbb{Z}$ is also continuous.  Since $(\Sigma_\epsilon, \mathbb{R}\times C)$ is connected, it follows that $h:  (\Sigma_\epsilon, \mathbb{R}\times C) \rightarrow \mathbb{Z}$  is constant, hence
	$h:  (\Sigma_\epsilon, \mathbb{R}\times C^1) \rightarrow \mathbb{Z}$ is constant as well, which proves that $(\Sigma_\epsilon, \mathbb{R}\times C^1)$ is connected.

	Now we are able to prove that $\Sigma_\epsilon$ contains the branch of minimal solutions of $(P_\epsilon)$, that is, if ${\lambda'} \in (0, \Lambda_\epsilon)$ and $\underline{u}_{{\lambda}'}$ is a minimal solution of $(P_{\epsilon})$ with $\lambda = {\lambda'}$, then $({\lambda'}, \underline{u}_{{\lambda'}}) \in \Sigma_\epsilon$. On the contrary, consider 
	$$ A = (0, {\lambda'}) \times \left\{u \in C_0^1(\overline{\Omega})_+ ~: ~ 0 < u < \underline{u}_{{\lambda'}} ~\mbox{in} ~\Omega,~ 0 > \frac{\partial u}{\partial\upsilon } > \frac{\partial \underline{u}_{{\lambda'}}}{\partial\upsilon} ~ \mbox{on} ~\partial \Omega\right\}$$
	an open and bounded set in $C_0^1(\overline{\Omega})_+$, where $\upsilon$ is the outward unit normal to $\partial \Omega$. Notice that $A \cap \Sigma_\epsilon \neq \emptyset$ and, by our  contradiction hypothesis, $\Sigma_\epsilon \cap \left(\{{\lambda'}\}\times [0, \underline{u}_{{\lambda'}}]\right)= \emptyset$.  Moreover, for $(\lambda,u) \in \Sigma_\epsilon \cap \overline{A}$ with $\lambda \in [0, {\lambda'})$ we have 
	\begin{equation}\label{pcf1}
		\begin{array}{l}
			-\Delta_p u - \lambda(u+\epsilon)^{-\delta} = \lambda u^q, \\
			-\Delta_p \underline{u}_{{\lambda'}}- \lambda(\underline{u}_{{\lambda'}}+\epsilon)^{-\delta} = {\lambda'} \underline{u}_{{\lambda'}}^q + ({\lambda'} - \lambda)(\underline{u}_{{\lambda'}} + \epsilon)^{-\delta},
	\end{array}\end{equation}
	where $\lambda u^q < {\lambda'} \underline{u}_{{\lambda'}}^q + ({\lambda'} - \lambda)(\underline{u}_{{\lambda'}} + \epsilon)^{-\delta}$ in $\Omega$ because  $0\leq u \leq \underline{u}_{{\lambda'}}$ in $\Omega$. Thus, by taking advantage of the proof of the Theorem 2.3 in \cite{Giaco}, we conclude from (\ref{pcf1}) that 
	$u < \underline{u}_{{\lambda'}}$ in $\Omega$. Therefore, $\Sigma_\epsilon \cap \partial A = \{(0,0)\}$, which contradicts the unboundedness and $C^1-$connectedness of $\Sigma_\epsilon$. Hence, $\Sigma_\epsilon$ contains the branch of minimal solutions of $(P_\epsilon)$.

	\textit{Part ii):}  We argue by contradiction. Suppose there exists a pair $(\lambda_*, u_*)$ of solution of the problem $(P_{\epsilon})$ with $\lambda_* > \Lambda_{\epsilon}$. 
	Without loss of generality, we can assume that $u_*$ is a minimal solution of $(P_{\epsilon})$ with $\lambda = \lambda_*$.
	
Consider  the open and bounded set in $C_0^1(\overline{\Omega})_+$ defined by
$$ A = (0, \lambda_*) \times \left\{u \in C_0^1(\overline{\Omega})_+ ~: ~ 0 < u < u_* ~\mbox{in} ~\Omega~\mbox{and}~ 0 > \frac{\partial u}{\partial\upsilon } > \frac{\partial u_*}{\partial\upsilon} ~ \mbox{on} ~\partial \Omega\right\}$$ 
and notice that $A \cap \Sigma_\epsilon \neq \emptyset$. Proceeding exactly as in \textit{Part-i)} one gets $\Sigma_\epsilon \cap \partial A = \{(0,0)\}$, again contradicting the unboundedness and connectedness of $\Sigma_\epsilon$.

\textit{Part $iii)$:} Let $\lambda' \in (0, \Lambda_\epsilon)$. In the following discussion, $\underline{u}_{\Lambda_\epsilon}$ denotes the minimal solution of $(P_\epsilon)$ with $\lambda = \Lambda_{\epsilon}$. If $(P_\epsilon)$ does not admit a solution for $\lambda = \Lambda_{\epsilon}$, then just replace  $\underline{u}_{\Lambda_\epsilon}$  with $\underline{u}_{\lambda''}$, where $\lambda'' \in (\lambda', \Lambda_{\epsilon})$.  Now, we argue by contradiction. Suppose that  $u \leq \underline{u}_{\Lambda_\epsilon}$, whenever $({\lambda'},u) \in \Sigma_\epsilon$. 
 In this case, it follows from the strong comparison principle \cite{Giaco} that $u < \underline{u}_{\Lambda_\epsilon}$ in $\Omega$ and ${\partial u}/{\partial\upsilon } > {\partial \underline{u}_{\Lambda_\epsilon}}/{\partial\upsilon}$ on $\partial \Omega$. Consider the open and bounded set 
 $$V = \left\{u \in C_0^1(\overline{\Omega})_+ ~: ~ u(x) < \underline{u}_{\Lambda_\epsilon}(x) ~\mbox{in} ~ \Omega ~~\mbox{and} ~~  \frac{\partial u}{\partial\upsilon }(x) > \frac{\partial \underline{u}_{\Lambda_\epsilon}}{\partial\upsilon}(x) ~ \mbox{on} ~\partial \Omega \right\}$$ and 
 $B = [0, {\lambda'}] \times V^c$. Cleary $\Sigma_\epsilon \cap B^c  \neq \emptyset$ and $\Sigma_\epsilon \cap B \neq \emptyset$, because $\Sigma_\epsilon$ bifurcates from infinity at $\lambda = 0$ and emanates from $(0,0)$.  On the other hand,  we have
 $$\partial B = \left(\{0, {\lambda'}\}\times \overline{V}^c\right) \cup \left([0, {\lambda'}] \times \partial V^c\right), $$ where 
 \begin{eqnarray*}
 	\partial V^c &=& \overline{V} \cap  \overline{V}^c \subseteq  \left\{u \in C_0^1(\overline{\Omega})_+ ~: ~ u(x) \leq \underline{u}_{\Lambda_\epsilon}(x) ~\mbox{in} ~~ \Omega ~~\mbox{and} ~~  \frac{\partial u}{\partial\upsilon }(x) \geq \frac{\partial \underline{u}_{\Lambda_\epsilon}}{\partial\upsilon}(x) ~ \mbox{on} ~\partial \Omega \right\} \\
 	& & \bigcap \left\{u \in C_0^1(\overline{\Omega})_+ ~: ~ u(x) \geq \underline{u}_{\Lambda_\epsilon}(x) ~\mbox{for some} ~x \in \Omega ~~\mbox{or} ~~  \frac{\partial u}{\partial\upsilon }(x) \leq \frac{\partial \underline{u}_{\Lambda_\epsilon}}{\partial\upsilon}(x) ~ \mbox{for some} ~ x \in \partial \Omega \right\} \\
 & \subseteq & \left\{u \in C_0^1(\overline{\Omega})_+ ~: ~ u(x) \leq \underline{u}_{\Lambda_\epsilon}(x) ~\mbox{in} ~ \Omega ~~\mbox{and} ~~ u(x) = \underline{u}_{\Lambda_\epsilon}(x) ~\mbox{for some} ~x \in \Omega \right. \\
  & & \left. \mbox{or} ~ \frac{\partial u}{\partial\upsilon }(x) = \frac{\partial \underline{u}_{\Lambda_\epsilon}}{\partial\upsilon}(x) ~ \mbox{for some} ~ x \in \partial \Omega \right\},
 \end{eqnarray*}
 which implies again  by Theorem 2.3 in \cite{Giaco} that $\Sigma_\epsilon \cap [0, {\lambda'}] \times \partial V^c = \emptyset$. Since  $\Sigma_\epsilon \cap \left(\{0, {\lambda'}\}\times \overline{V}^c \right)= \emptyset$, we have $\Sigma_\epsilon \cap \partial  B= \emptyset, $ contradicting the $C^1$-connectedness of $\Sigma_\epsilon$. From this, the proof of item$-iii)$ is established.

\textit{Part $iv
	)$:} Let $\kappa > 0$ small, $0 < \epsilon_1 < \epsilon_2$ and denote by $\underline{u}_{i}$ the minimal solution of the problem
$(P_{\epsilon_i})$ with $\lambda = \Lambda_{\epsilon_i} - \kappa$, $i=1,2$. 
In this case, $\underline{u}_{1}$ is  a supersolution of 
\begin{equation}\label{ccc1}
\left\{\begin{array}{l}
-\Delta_p u = (\Lambda_{\epsilon_{1}}-\kappa)\left((u+\epsilon_2)^{-\delta} +	u^q\right) ~~\mbox{in}~\Omega  \\
u = 0 ~~\mbox{on} ~\partial{\Omega}.
\end{array}\right.
\end{equation}
and $v = 0$ is subsolution. So, (\ref{ccc1}) admits a positive solution in $[0, \underline{u}_{1}] $ and by arbitrariness of $\kappa$ we conclude that $ \Lambda_{\epsilon_{1}} \leq \Lambda_{\epsilon_2}$.
\end{proof}
\fim

Below, we present some of the possible bifurcation diagrams of $\Sigma_\epsilon$.

\begin{figure}[h!]
\centering
\begin{overpic}[width=0.8\textwidth]{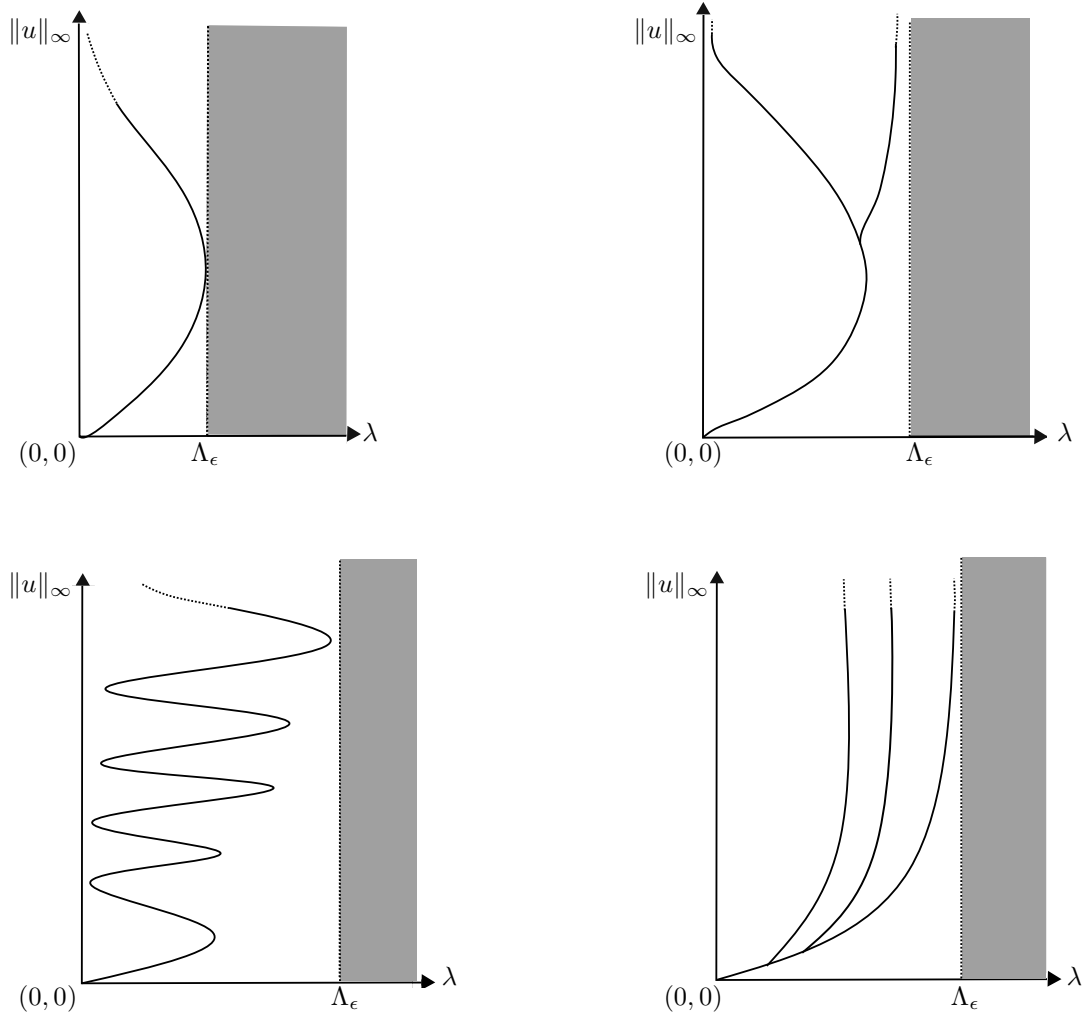}
	\put(12,53){$\Lambda_\epsilon$}
	\put(82,53){ $\Lambda_\epsilon$}
	\put(29,55){{$\lambda$}}
	\put(-7,40){{ $\|u\|_\infty$}}
	\put(-7,95){{ $\|u\|_\infty$}}
	\put(98,1){{ $\lambda$}}
	\put(57,40){{$\|u\|_\infty$}}
	\put(55,95){{ $\|u\|_\infty$}}
	\put(97,55){{ $\lambda$}}
	\put(86.5,-1){{ $\Lambda_\epsilon$}}
	\put(25,-1){{ $\Lambda_\epsilon$}}
	\put(36,1){{ $\lambda$}}
	\put(-6,-1){{ $(0,0)$}}
	\put(58,-1){{ $(0,0)$}}
	\put(-6,53){{ $(0,0)$}}
	\put(58,53){{ $(0,0)$}}
\end{overpic}
\caption{Possible bifurcation diagram for  $\Sigma_{\epsilon}$ and non-existence regions}
\end{figure}
\newpage

\begin{lemma}\label{add} Let $g \in C(\overline{\Omega}\times\mathbb{R})$ be a non-negative function and $\{\epsilon_k\}_{k=1}^{\infty} \subset (0,1)$ a  sequence satisfying $\epsilon_k \rightarrow 0^+$. If  $u_k \in W_0^{1,p}(\Omega) \cap C(\overline{\Omega})$ is a solution of 
	\begin{equation}\label{geral1}
	\left\{\begin{array}{l}
	-\Delta_p u = \lambda_k\left[(u+\epsilon_k)^{-\delta} + g(x,u)\right] ~~\mbox{in} ~~\Omega, \\
	u > 0 ~~\mbox{in} ~\Omega, ~~u = 0 ~~\mbox{on} ~\partial\Omega
	\end{array}\right.
	\end{equation}
such that	$0 <\displaystyle\inf_k \lambda_k \leq  \lambda_k \leq \lambda$ and $0 < u_k(x) \leq h(x)$, for some $\lambda > 0$ and $h \in C_0(\overline{\Omega})_+$, then there exists $(\lambda_*,u_*) \in \mathbb{R}^+ \times \left(  W_\mathrm{loc}^{1,p}(\Omega) \cap C_0(\overline{\Omega}) \right)$ such that 
	$$\lambda_k \rightarrow \lambda_* ~~\mbox{and} ~~~u_k \rightarrow u_* ~~\mbox{in} ~ C(\overline{\Omega}) ~~\mbox{and} ~~W_\mathrm{loc}^{1,p}(\Omega), $$ up to a subsequence.
	Moreover, $u_*$	solves 
	\begin{equation}\label{geral2}
	\left\{\begin{array}{l}
	-\Delta_p u = \lambda_*\left[u^{-\delta} + g(x,u)\right] ~~\mbox{in} ~~\Omega, \\
	u > 0 ~~\mbox{in} ~\Omega, ~~u = 0 ~~\mbox{on} ~\partial\Omega.
	\end{array}\right.
	\end{equation}
	
\end{lemma}

\begin{proof} By our hypothesis on the sequence $\{\lambda_k\}_{k=1}^{\infty}$, there exists $\lambda_* > 0$ such that $\lambda_k \rightarrow \lambda_*,$  up to a subsequence. Hence, for $\tau>0$ small we have $0 < \lambda_* - \tau < \lambda_k < \lambda_* + \tau$, for every $k$ enough large. From this inequality and classical weak comparison principles one obtains $\omega_{\lambda_*-\tau,1} \leq u_k$ in $\Omega$, where $\omega_{\lambda_*-\tau,1} \in C_0^1(\overline{\Omega})$ is the only solution of (\ref{sp}) with $\lambda = \lambda_*-\tau$ and $\epsilon = 1$.
	
	Consider a sequence $(\Omega_i)$ of open sets in $\Omega$ such that $\Omega_i \subset  \subset \Omega_{i+1}$, $\bigcup_i \Omega_i = \Omega$ and define $\gamma_i = \displaystyle\min_{\overline{\Omega}_i}\omega_{\lambda_*-\tau,1},$ for each $i \in \mathbb{N}$. Using that $g(x,u_k) \leq \displaystyle\max_{\overline{\Omega}\times[0,\|h\|_{\infty}]} g(x,t) $ in $\Omega$ and testing the problem (\ref{geral1}) against $\varphi = (u_k - \gamma_1)^+$, we obtain
	\begin{eqnarray*}
		\displaystyle\int_{[u_k \geq \gamma_1]}|\nabla u_k|^pdx  =  \lambda_k\displaystyle\int_{[u_k \geq \gamma_1]}\left[(u_k+\epsilon_k)^{-\delta} + g(x,u_k)\right](u_k - \gamma_1)^+dx \leq C_1,
	\end{eqnarray*}
	where $C_1 > 0 $ is a real constant independent of $k$. Hence, the sequence $\{u_k\}_{k=1}^{\infty}$ is bounded in $W^{1,p}(\Omega_1)$ and there exists $u_{\Omega_1} \in W^{1,p}(\Omega_1)$ and a subsequence $\{u_{k_j^1}\}$ of $\{u_k\}$ such that
	$$
	\left\{
	\begin{array}{l}
	u_{k_j^1} \rightharpoonup u_{\Omega_1}  \ \mbox{weakly in}  \ W^{1,p}(\Omega_1) \  \mbox{and strongly in} \ L^s(\Omega_1)  \ \mbox{for} \ 1 \leq s < p^* \\
	u_{k_j^1} \rightarrow u_{\Omega_1} \ \
	a.e.  \ \mbox{in } \ \Omega_1.
	\end{array}
	\right.
	$$
	
	Proceeding as above  through a diagonal argument we can obtain subsequences $\{u_{k_j^i}\}$ of $\{u_k\}$, with $\{u_{k_j^{i + 1}}\} \subset \{u_{k_j^{i}}\}$, and functions $ u_{\Omega_i} \in W^{1,p}(\Omega_i)$ such that
	$$
	\left\{
	\begin{array}{l}
	u_{k_j^i} \rightharpoonup u_{\Omega_i}\ \mbox{weakly in }  \ W^{1,p}(\Omega_i) \  \mbox{and strongly in} \ L^s(\Omega_i) \ \mbox{for} \ 1 \leq s < p^*  \\
	u_{k_j^i} \rightarrow u_{\Omega_i} \ \
	a.e. \ \mbox{in}   \ \Omega_i.
	\end{array}
	\right.
	$$
	
	By construction,  we have $ u_{{\Omega_{i+1}}{\big|_{\Omega_i}}} = u_{\Omega_i}. $ Hence,  $$u_* := \left\{
	\begin{array}{lll} u_{\Omega_1} & \mbox{in} &  \Omega_1, \\ u_{\Omega_{i+1}} & \mbox{in}&  \Omega_{i+1}\backslash \Omega_i
	\end{array}
	\right.$$ belongs to $  W_{\mathrm{loc}}^{1,p}(\Omega)$ and satisfies $\omega_{\lambda_*-\tau,1} \leq u_* \leq h$ in $\Omega$.
	
	We claim that $u_*$ is a solution of (\ref{geral2}). Indeed, by taking $\varphi \in C_c^{\infty}(\Omega)$ and using Theorem 2.1 of Boccardo and Murat \cite{Murat}, we obtain	\begin{equation}\label{11}
	\displaystyle\int_{\Omega}|\nabla u_k|^{p-2}\nabla u_k\nabla \varphi dx \rightarrow \displaystyle\int_{\Omega}|\nabla u|^{p-2}\nabla u\nabla \varphi dx,
	\end{equation}
	up to a subsequence. On the other hand, it follows from the convergence $u_k \rightarrow u_*$ a.e in $\Omega$,  continuity of $g$, uniform boundedness of $\{u_k\}_{k=1}^{\infty}$ and Lebesgue dominated convergence theorem that
	\begin{equation}\label{12}
	\lambda_k\displaystyle\int_{\Omega} \left[(u_k+\epsilon_k)^{-\delta} + g(x,u_k)\right]\varphi dx \rightarrow \lambda_*\displaystyle\int_{\Omega}\left[u_*^{-\delta} + g(x,u_*)\right]\varphi dx.
	\end{equation}
	Therefore, combining (\ref{11}) and (\ref{12}) one has 
	$$\displaystyle\int_{\Omega}|\nabla u_*|^{p-2}\nabla u_* \varphi dx = \lambda_*\displaystyle\int_{\Omega}\left[u_*^{-\delta} + g(x,u_*)\right]\varphi dx, $$
	for all $\varphi \in C_c^{\infty}(\Omega)$, which proves that $u_*$ solves (\ref{eq2}).
	
	To conclude that  $(\lambda_k, u_k) \rightarrow (\lambda_*, u_*)$ in $\mathbb{R}\times C(\overline{\Omega})$ as well,  we just need to combine $L^{\infty}$-uniform bound of $\{u_k\}_{k=1}^{\infty}$ and Arzel�-Ascoli theorem with Theorem 1.8 of \cite{JKS}.
\end{proof}
\fim
\vspace{0.4cm}

\begin{lemma}\label{L1.9}
Let  $B_R(0,0)\subset \mathbb{R} \times C_0(\overline{\Omega})$ be the ball centered at $(0,0)$ with radius $R$, $\epsilon \in (0,1)$ and  $(\lambda_{\epsilon}, u_{\epsilon}) \in \left((0, \infty) \times ( W_0^{1,p}(\Omega)) \cap C_0(\overline{\Omega})\right) \cap \overline{B}_R(0,0)$ be a pair of solution of
	\begin{equation}\label{geral}
	\left\{\begin{array}{l}
	-\Delta_p u = \lambda\left[(u+\epsilon)^{-\delta} + u^q\right] ~~\mbox{in} ~~\Omega, \\
	u > 0 ~\mbox{in} ~\Omega, ~~u = 0 ~~\mbox{on} ~\partial\Omega,
	\end{array}\right.
	\end{equation}		
	satisfying $\|(\lambda_\epsilon, u_\epsilon)\|_\infty > \varrho$, for some $\varrho \in (0,R)$. Then, there exist positive constants $\mathcal{K}_1 = \mathcal{K}_1(R, \varrho)$ and $\mathcal{K}_2 = \mathcal{K}_2(r, R)$  such that
	
	\begin{equation}\label{7}
	\lambda_{\epsilon}^{\frac{1}{p-1}}\mathcal{K}_1(R,\varrho) \phi_1 \leq u_{\epsilon} \leq r + \lambda_{\epsilon}^{\frac{1}{p-1}}\mathcal{K}_2(r, R)^{\frac{1}{p-1}}e_p ~~ \mbox{in} ~\Omega,
	\end{equation}
	for each $r \in (0, R]$ fixed, where $e_p$ is defined in (\ref{torsion}).
\end{lemma}
\begin{proof}
	To prove the first inequality in (\ref{7}), we set
	$$\mathcal{K}_2(r, R) = \max\left\{t^{-\delta} + t^q: ~  r \leq t \leq R + 1, x \in \overline{\Omega}\right\}, $$ where $r$ is a fixed  number on  $(0, R ]$, and $\mathcal{O}_r = \{ x \in \Omega : u_\epsilon > r\}$. Then, it follows from the definition of $\mathcal{K}_2$ that
	\begin{eqnarray*}
		-\Delta_p \Big(  r + \lambda_{\epsilon}^{\frac{1}{p-1}}\mathcal{K}_2(r, R)^{\frac{1}{p-1}}e_p \Big) = \lambda_{\epsilon}\mathcal{K}_2(r, R) \geq {\lambda_{\epsilon}} \left(u_\epsilon + \epsilon)^{-\delta} + u_\epsilon^q
		\right) \geq -\Delta_p u_\epsilon~~ \mbox{in} ~ \mathcal{O}_r .
	\end{eqnarray*}
	Since $  r + \lambda_{\epsilon}^{\frac{1}{p-1}}\mathcal{K}_2(r, R)^{\frac{1}{p-1}}e_p - u_{\epsilon} = \lambda_{\epsilon}^{\frac{1}{p-1}}\mathcal{K}_2(r, R)^{\frac{1}{p-1}}e_p \geq 0 $ on $\partial\mathcal{O}_r$, the claim is valid in $\mathcal{O}_r$ by classical comparison principles. In the complementary of $\mathcal{O}_r$, the inequality is obvious. 
	
	To show the first inequality in (\ref{7}), we start by proving that
	$$
	\lambda_{\epsilon} >  C_* := \min\left\{ \frac{1}{\mathcal{K}_2(\varrho/4, R)}\left(\frac{\varrho}{4\|e_p\|_\infty}\right)^{p-1}, \frac{\varrho}{4}\right\} .
	$$
	In fact, otherwise by taking $r = \varrho/4$ in the second inequality in (\ref{31}) we would have $(\lambda_\epsilon,  u_{\epsilon} ) \in B_{3\varrho/ 4}(0,0)\subset \mathbb{R}\times C(\overline{\Omega})$, which contradicts the fact that $\|(\lambda_{\epsilon}, u_\epsilon)\|_\infty >  \varrho . $
	
	Now, let us define $\underline{u}_{\epsilon} = \lambda_{\epsilon}^{\frac{1}{p-1}}\mathcal{K}_1(R, \varrho)\phi_1$, where $\mathcal{K}_1(R, \varrho)$  will be chosen later. It follows from Picone's inequality that
	\begin{eqnarray}\label{8}
	0 & \leq & \displaystyle\int_{\Omega}\left[ |\nabla\underline{u}_{\epsilon}|^{p-2}\nabla\underline{u}_{\epsilon}\nabla \left(\frac{\underline{u}_{\epsilon}^p - {u}_{\epsilon}^p}{\underline{u}_{\epsilon}^{p-1}}\right)^+ -  |\nabla{u}_{\epsilon}|^{p-2}\nabla{u}_{\epsilon}\nabla \left(\frac{\underline{u}_{\epsilon}^p - {u}_{\epsilon}^p}{{u}_{\epsilon}^{p-1}}\right)^+\right]dx \nonumber\\
	& \leq & \lambda_{\epsilon} \displaystyle\int_{\Omega} \left[\frac{\lambda_1\mathcal{K}_1^{p-1}\phi_1^{p-1}}{(\lambda_{\epsilon}^{1/(p-1)}\mathcal{K}_1\phi_1)^{p-1}} - \frac{(u_\epsilon+\epsilon)^{-\delta} + u_\epsilon^q}{u_\epsilon^{p-1}}\right]\left(\underline{u}_{\epsilon}^p - u_{\epsilon}^p\right)^+ dx \nonumber\\
	& = & \lambda_{\epsilon}\displaystyle\int_{\Omega}\left[ \frac{\lambda_1}{\lambda_{\epsilon}} - \frac{(u_\epsilon + \epsilon)^{-\delta} + u_\epsilon^q}{u_{\epsilon}^{p-1}}\right]\left(\underline{u}_{\epsilon}^p - u_{\epsilon
	}^p\right)^+ dx.	\end{eqnarray}

	Since $((t+1)^{-\delta} + t^q)/t^{p-1} \rightarrow +\infty$ as $t \rightarrow 0^+$, for $\tilde{K} > \max\{R, \lambda_1/C_*\}$ given we can find $a > 0$ such that $(t+1)^{-\delta} + t^q \geq \tilde{K}t^{p-1},$ for all  $0 < t < a$. Hence, for  $\mathcal{K}_1(R, \varrho) = {a}/\left({2
		\tilde{K}^{\frac{1}{p-1}}\|\phi_1\|_{\infty}}\right)$ the first inequality in (\ref{7}) holds. 
	Indeed, if  $\vert [\underline{u}_{\epsilon} > u_\epsilon] \vert >0$
	 then
	$${u}_{\epsilon}  \leq \underline{u}_{\epsilon}  \leq \frac{a}{2} ~\mbox{on } [\underline{u}_{\epsilon} > u_\epsilon].$$	
	Therefore, going back to (\ref{8}) and using that ${\lambda_1}/{\lambda_\epsilon} \leq {\lambda_1}/{C_*}$, we get
	\begin{eqnarray*} 0 &\leq & \lambda_{\epsilon}\displaystyle\int_{\Omega}\left[ \frac{\lambda_1}{\lambda_{\epsilon}} - \frac{(u_\epsilon+ \epsilon)^{-\delta} +  u_\epsilon^q}{u_{\epsilon}^{p-1}}\right]\left(\underline{u}_{\epsilon}^p - u_{\epsilon}^p\right)^+ dx \\
		&  \leq & \lambda_{\epsilon}\displaystyle\int_{\Omega}\left[ \frac{\lambda_1}{C_*} - \frac{\tilde{K}u_{\epsilon}^{p-1}}{u_{\epsilon}^{p-1}}\right]\Big(\underline{u}_{\epsilon}^p - u_{\epsilon}^p\Big)^+ dx < 0,
	\end{eqnarray*}
	which is an absurd. Hence, $\lambda_{\epsilon}^{\frac{1}{p-1}}\mathcal{K}_1(R,\varrho) \phi_1 \leq u_{\epsilon}$ in $\Omega$ and the inequality (\ref{7}) is proved.
\end{proof}
\vspace{0.3cm}

\textbf{Proof Theorem \ref{T1}}
 Our proof will be based again on the Lemma \ref{L9}.  Initially, notice that $(0,0) \in \Sigma_\epsilon$, for all $0 < \epsilon < \varepsilon_1$, whence   
	 the pair $(0,0)$ fulfills the first condition of the mentioned lemma. To prove that the second condition in Lemma \ref{L9} is also satisfied, let $B_R(0,0) \subset \mathbb{R}\times C_0(\overline{\Omega})$ be the ball centered at $(0,0)$ with radius $R>0$,   $\{\epsilon_n\}_{n=1}^{\infty} \subset  (0, \varepsilon_1)$ a sequence such that $\epsilon_n \rightarrow 0^+$, and $\{(\lambda_k, u_k)\}_{k=1}^{\infty}$ a sequence in $\left(\displaystyle\bigcup_{n=1}^{+\infty}\Sigma_{\epsilon_n}\right) \cap B_R(0,0)$. We have three cases to consider:
	 \begin{itemize}
	 	\item[a)]  an infinite amount terms of the sequence $\{(\lambda_k, u_k)\}_{k=1}^{\infty}$ belongs to some $\Sigma_{\epsilon_n}$.
	 	\item[b)] $(0,0)$ is a limit point of $\{(\lambda_k, u_k)\}_{k=1}^{\infty}$.
	 	\item[c)] $\{(\lambda_k, u_k)\}_{k=1}^{\infty}$ has terms on infinite amount of $\Sigma_{\epsilon_n}$ and $(0,0)$ is not a limit point of this sequence.
	 \end{itemize} 	 
	 
	 If $a)$ occurs,  by using Arzel�-Ascoli theorem, we get a convergent subsequence in the $\mathbb{R}\times C_0(\overline{\Omega})-$topology. If  condition $b)$ holds, naturally we have a convergent subsequence as well. In the case of $c)$ be true, we can assume without loss of generality that $(\lambda_k, u_k) \in \Sigma_{\epsilon_k}$ and $\varrho \leq |(\lambda_k, u_k)|_\infty \leq R$, for some $\varrho > 0$ and for all $k \in \mathbb{N}$.
	Thus we are able to use Lemma \ref{L1.9} to obtain positive constants $\mathcal{K}_1= \mathcal{K}_1(R,\varrho)$ and $\mathcal{K}_2 = \mathcal{K}_2(r, R)$ such that 
	\begin{equation}\label{31}
	\lambda_{k}^{\frac{1}{p-1}}\mathcal{K}_1(R,\varrho) \phi_1 \leq u_{k} \leq r + \lambda_{k}^{\frac{1}{p-1}}\mathcal{K}_2(r, R)^{\frac{1}{p-1}}e ~~ \mbox{in} ~\Omega,
	\end{equation}
	for each $r \in (0, R]$ fixed.
	 
	 
	 Suppose that $\lambda_k \rightarrow \lambda \geq 0$. If $\lambda = 0$, then by (\ref{31}) we have $(\lambda_k, u_k) \rightarrow (0,0)$ in $\mathbb{R}\times C_0(\overline{\Omega})$, which contradicts the fact that $(0,0)$ is not a limit point of the sequence $\{(\lambda_k,u_k)\}_{k=1}^{\infty}$.  Therefore, $0 <\displaystyle\inf_k \lambda_k \leq  \lambda_k \leq R$ for all $k$ sufficiently large. From this and (\ref{31}), the existence of the subsequence convergent of $\{(\lambda_k,u_k)\}_{k=1}^{\infty}$ is a consequence of Lemma \ref{add}.
	 
	 Therefore, from  Lemma \ref{L9}, Lemma \ref{L11} and Proposition \ref{P2}
	  we obtain that
	 $$\Sigma' := \displaystyle\lim_{n \rightarrow \infty}\sup \Sigma_{\epsilon_n}$$ is unbounded, closed, connected and joins $(0,0)$ to $(0, +\infty)$. Moreover, by Proposition \ref{P2} we also have $\overline{\displaystyle\mbox{Proj}_{\mathbb{R}^+} \Sigma'} := [0, \Lambda^*] \subset [0, \lambda_1(\zeta+1)^{\delta}\zeta^{p-1}]$ and $\Lambda^* \leq \Lambda_{\epsilon}$ (see item$-iv$ in Proposition \ref{P2}).
	 
	 Let us prove that $\Sigma:= \Sigma'\backslash \{(0,0)\}$ has the properties stated in the theorem. It is a direct consequence of the Lemma \ref{add}  and the construction of $\Sigma$ that $\Sigma$ is formed by solutions of $(P)$.
	 
	 Next, let us show that $\Sigma$ contains the branch of minimal solutions of $(P)$. In fact, assume $\lambda_* \in \displaystyle\mbox{Proj}_{\mathbb{R}} \Sigma $, let $(\lambda_*, u_*) \in \mathbb{R}_+ \times C_0(\overline{\Omega})$ be a pair of solution of $(P)$  and consider the iterative process
	\begin{equation}\label{itsing}
	\left\{\begin{array}{l}
	-\Delta_p u_n - \lambda_* u_n^{-\delta} = \lambda_* u_{n-1}^q ~~ ~\mbox{in} ~~\Omega, \\
	u_0 = 0, ~~~~ ~u_n \in W^{1,p}_{\mathrm{loc}}(\Omega) \cap C_0
	(\overline{\Omega})
	\end{array}\right.
	\end{equation} 

It is clear that $u_0 \leq u_1 $ in $\Omega$. By induction, we assume  $u_{n-1} \leq u_n $ in $\Omega$ and let us prove that  
$u_{n} \leq u_{n+1}$ in $\Omega$. Indeed, 
$$ \begin{array}{l}
	-\Delta_p u_n  =  \lambda_* u_n^{-\delta} + \lambda_* u_{n-1}^q \\
	-\Delta_p u_{n+1}  =  \lambda_* u_{n+1}^{-\delta} + \lambda_* u_{n}^q \geq \lambda_* u_{n+1}^{-\delta} + \lambda_* u_{n-1}^q,
\end{array}
$$
that is, $u_n$ is a solution and $u_{n+1}$ is a supersolution of 
$$ -\Delta_p u = \lambda_* u^{-\delta} + \lambda_* u_{n-1}^q ~~\mbox{in} ~\Omega, ~~u|{\partial \Omega} = 0, $$ respectively. So, we can apply the comparison principle of \cite{ZAMP} to conclude that $u_n \leq u_{n+1}$ in $\Omega$, as claimed. 

 Analogously,  we can show that $0 < u_n \leq u_*$ in $\Omega$, for all $n \in \mathbb{N}$. Since $0 < u_1 \leq u_n \leq u_*$ for all $n \geq 1$, we are able to employ the same steps of proof of Lemma \ref{add} to ensure the existence of a solution $\underline{u}_* \in W^{1,p}_{\mathrm{loc}}(\Omega) \cap C_0(\overline{\Omega})$ of $(P)$ such that $u_n \rightarrow \underline{u}_*$ in $W^{1,p}_{\mathrm{loc}}(\Omega)$ and in $C_0(\overline{\Omega})$, up to a subsequence. Furthermore,  the construction of $\underline{u}_*$ assures us that this must be the minimal solution of $(P)$ with $\lambda = \lambda_*$.
	
Finally, we will show that $(\lambda_*, \underline{u}_*) \in \Sigma $. To this end, let us consider $\epsilon_k \searrow 0^+$ as $k \rightarrow +\infty$ and denote by $\underline{u}_{\epsilon_k}$ the minimal solution of $(P_{\epsilon_k})$ with $\lambda = \lambda_*$. Once again by monotonic iteration and the comparison principle in \cite{ZAMP}, we have 
\begin{equation}\label{add2}
\underline{u}_{1} \leq \underline{u}_{\epsilon_k} \leq \underline{u}_*
\end{equation} for all $\epsilon_k \in (0,1]$.  It follows from  Lemma \ref{add} and inequalities (\ref{add2}) that $\underline{u}_{\epsilon_k} \rightarrow \underline{u}_*$ as $k \rightarrow \infty$. Since $(\lambda_*, \underline{u}_{\epsilon_k}) \in \Sigma_{\epsilon_k}$, the construction of $\Sigma$ provides $(\lambda_*, \underline{u}_*)\in \Sigma$.
	 
	 Now let us verify that item$-ii)$ holds. On the contrary, we could find a pair \linebreak 
	$(\lambda_*,u_*) \in \mathbb{R}\times \left(W^{1,p}_{\mathrm{loc}}(\Omega)\cap C_0(\overline{\Omega})\right)$ of solution of the problem $(P)$ with $\lambda_* > \Lambda^*$. 
	 
	 Let $\epsilon_k \searrow 0^+$ as $k \rightarrow +\infty$.  Given $\tau = (\lambda_* - \Lambda^*)/2$, we can apply Lemma \ref{L11} to find some $k_0 \in \mathbb{N}$ such that $\Sigma_{\epsilon_k} \subset V_{\tau}(\Sigma')$, for all $k > k_0$. In particular, $\Lambda_{\epsilon_k} \leq \Lambda^* + \tau < \lambda_*$ for all $k > k_0$, where $\Lambda_{\epsilon_k}$ is the threshold parameter for the existence of solutions of $(P_{\epsilon_k})$. Let us fix $k > k_0$, $\hat{\lambda} \in (0,\Lambda_{\epsilon_k}]$ and consider $\underline{\hat{u}}_{\epsilon_k} \in W_0^{1,p}(\Omega) \cap C_0(\overline{\Omega})$  the minimal solution of $(P_{\epsilon_k})$ with $\lambda = \hat{\lambda}$. Since $u_*$ is a supersolution of the problem $(P_{\epsilon_k})$ with $\lambda = \hat{\lambda}$, once again by monotonic iteration we can conclude that $\underline{\hat{u}}
_{\epsilon_k} \leq u_*$ in $\Omega$. So, about the problem  
\begin{equation}\label{38}\left\{
\begin{array}{l}
-\Delta_pu = \lambda_*((u+\epsilon_k)^{-\delta} + u^q) ~~ \mbox{in} ~~\Omega,\\
u > 0 ~\mbox{in} ~\Omega,~~u = 0 ~\mbox{on} ~\partial \Omega,
\end{array}
\right.
\end{equation}
we  can summarize the following facts:
\begin{itemize}
	\item $\underline{\hat{u}}
	_{\epsilon_k}$ is a subsolution of (\ref{38});
	\item $u_*$ is a supersolution of (\ref{38});
	\item $0 <\underline{\hat{u}}
	_{\epsilon_k}\leq u_*$ in $\Omega$.
\end{itemize}
Hence, we are able to apply Theorem 2.4 of \cite{Loc} to get a $W_0^{1,p}(\Omega)-$solution of (\ref{38}) in $[\underline{\hat{u}}
_{\epsilon_k},u_*]$, which contradicts the fact that (\ref{38}) does not admits any solution since $\lambda_* \notin \mbox{Proj}_{\mathbb{R}} \Sigma_{\epsilon_k}$. This proves item$-ii)$. 

Regarding item$-iii)$, the  multiplicity for $\lambda > 0$  small follows from the facts that $\Sigma$ is connected and $\lambda =0$ is a bifurcation value of $\Sigma$ from the infinity and from the trivial solution. 
 Indeed, let $C_* > 0$ be a positive
 constant such that $(P)$ admits at most a positive solution satisfying $\|u\|_\infty \leq C_*$ (such constant exists by Lemma \ref{L7}). If for some $\check{\lambda}> 0$ small enough  $(P)$ does not admit two distinct solutions in $\Sigma$, then we can define the open set $U = (0,\check{\lambda})\times V$, where $V := \{u \in C(\overline{\Omega}) ~: ~\|u\|_\infty > C_* \}$,
 and conclude that $U^c \cap \Sigma \neq \emptyset$ (since $(0,0) \in \overline{\Sigma}$) and $U \cap \Sigma \neq \emptyset$ (because $\lambda=0$ is a bifurcation value of $\Sigma$ from the infinity). However $\partial U \cap \Sigma = \emptyset$  because 
 $ \{0, \check{\lambda}\}\times \overline{V} = \emptyset $, due to our contradiction assumption,  and  $ [0, \check{\lambda}]\times \partial{V}  = \emptyset$, because $\|\underline{u}_{\lambda}\|_\infty < C_*$ for any $\lambda \leq \check{\lambda}$ since we are supposing $\check{\lambda}$ enough small. This contradiction leads us to conclude that there exists $\Lambda_* > 0$ such that $ (P)$ admits at least two solutions for $\lambda \in (0, \Lambda_*)$.

 In the particular case, when $\delta \in (0,1)$, the proof of the existence of at least two solutions for $\lambda \in (0, \Lambda^*)$ is obtained by  redoing the proof of item$-iii)$ in Proposition \ref{P2} and noting that in this case any continuous solution of $(P)$ belongs to $C_0^1(\overline{\Omega})$ (see Theorem B.1 in \cite{Giaco}).
\fim

\vspace{0.4cm}

Our goal from now on is to establish the proof of Theorem \ref{T2},  which is essentially inspired by \cite{Azizieh}, see also \cite{Chen} and \cite{Fig} . For this, we need to introduce some definitions and preliminary results. 

\begin{lemma}(\cite{Azizieh}, Lemma 4.1)\label{L35}
	 Let $\Omega \subset \mathbb{R}^N$ be a convex and bounded domain with  $C^2$ boundary. Then 
	 $$\{x \in \mathbb{R}^N ~: ~x = y + t\nu(y), ~0 < t < 2\rho\} \subset \Omega,$$
	 for some $\rho > 0$, where $\nu(y)$ denotes the inward unit normal to $\partial \Omega$ at $y$.	 
\end{lemma}
 Let $\nu$ be a direction in $\mathbb{R}^N$ with $|\nu| = 1$. For $\lambda \in \mathbb{R}$ we set
 \begin{itemize}
 	\item $ T_{\lambda}^{\nu} = \{x \in \mathbb{R}^N ~: ~x\cdot \nu = \lambda\}$
 	\item $a(\nu) = \displaystyle\inf_{x \in \Omega} x\cdot \nu $
 	\item $ \Omega_{\lambda}^{\nu} = \{x \in \mathbb{R}^N ~: ~x\cdot \nu < \lambda\},$  which is nonempty for $\lambda > a(\nu)$
 	\item $x_{\lambda}^{\nu} = R_{\lambda}^{\nu}(x) = x + 2(\lambda - x\cdot \nu)\nu$, the reflection of $x \in \mathbb{R}^N$ 
 	 through the hyperplane $ T_{\lambda}^{\nu}$
 	\item  $(\Omega_{\lambda}^{\nu})' = R_{\lambda}^{\nu}(\Omega_{\lambda}^{\nu})$, the reflection of $\Omega_{\lambda}^{\nu}$ 
 	through $ T_{\lambda}^{\nu}$
 	\item $ \Lambda_1(\nu) = \{ \lambda > a(\nu) ~: ~ \forall \mu \in (a(\nu), \lambda) ~\mbox{ none of conditions (a)
 		and (b)) holds }\}$, where the conditions $(a)$ and $(b)$ are the following:
 	\begin{itemize}
 		\item[a)] $(\Omega_{\lambda}^{\nu})'$ becomes internally tangent to $\partial \Omega$
 		\item[b)] $ T_{\lambda}^{\nu}$ is orthogonal to $\partial \Omega$
 	\end{itemize}
 \item $\lambda_1(\nu) = \sup \Lambda_1(\nu)$
 \end{itemize} 
 \begin{lemma}\label{L36}[\cite{Azizieh}, Lemma 4.2]  Let $\Omega \subset \mathbb{R}^N$ be a convex and bounded domain with  $C^2$ boundary and $\rho$ given in Lemma \ref{L35}. Then 
 	$$ \displaystyle\inf_{x \in \partial \Omega} dist(x, T_{\lambda_1(\nu(x)))}^{\nu(x)}) \geq \rho > 0.$$
 \end{lemma}

\begin{lemma}\label{L37}[\cite{Dam}, Theorem 1.5] Let $\Omega$ be a bounded smooth domain in $\mathbb{R}^N$, $N \geq 2$, $1 < p < \infty$, $f:[0, \infty) \rightarrow \mathbb{R}$ a continuous function which is strictly positive and locally Lipschitz
	continuous in $(0, \infty)$ and $u \in C^1(\overline{\Omega})$ a weak solution of
	$$\left\{\begin{array}{ll} 
	-\Delta_p u = f(u) & \mbox{in} ~\Omega \\
	u > 0 & \mbox{in} ~\Omega \\
	u = 0 & \mbox{on} ~\partial \Omega. \end{array} \right.$$
For any direction $\nu $ and for $\lambda$ in the interval $(a(\nu), \lambda_1(\nu)]$,  we have 
$$u(x) \leq u(x_{\lambda}^{\nu}), ~~~\forall  x \in \Omega_{\lambda}^{\nu}. $$
\end{lemma}

Now we are able to proof Theorem \ref{T2}, which follows similar strategy considered in \cite{Azizieh}, with minor changes. However, for the reader convenience, we include the details here.

\vspace{0.3cm}

\textbf{Proof of Theorem \ref{T2}:}  We argue by contradiction, that is, let us assume that there exists $\check{\lambda} \in (0, \Lambda^*]$ being a bifurcation parameter of $\Sigma$ from infinity. Then, by the construction of $\Sigma$, there would exist a subsequence of index $\mathbb{N}^\prime \subset \mathbb{N}$, a numerical sequence $\{\epsilon_n\}_{n \in \mathbb{N}^\prime}$ such that $\epsilon_n \searrow 0 $, and pairs $(\lambda_n, u_n) \in \Sigma_{\epsilon_n}$ satisfying 
$$\left\{\begin{array}{l}
\lambda_n \rightarrow \check{\lambda} ,\\
\|u_n\|_\infty \rightarrow \infty.
\end{array}\right.$$

\noindent\textbf{Claim 1:} For each $n \in  \mathbb{N}^\prime$, there exists a global maximum point $\tau_n \in \Omega$ of $u_n$ (that is, $u_n(\tau_n) = \|u_n\|_\infty)$ such that 
$\mbox{dist}(\tau_n, \partial \Omega) \geq \rho$. 
\vspace{0.2cm}

\noindent\textbf{Proof of claim 1:} Assume by contradiction that every global maximum point $\tau$ of $u_n$  satisfies $\mbox{dist}(\tau, \partial \Omega) < \rho- \epsilon$, for some $\epsilon \in (0, \rho)$. By fixing $\check{\tau}$ a such maximum and considering $\check{x} \in \partial \Omega$ the nearest point of $\partial \Omega$ from $\check{\tau}$, we have that $\mbox{dist}(\check{\tau}, \check{x}) = \mbox{dist}(\check{\tau}, \partial \Omega) < \rho -\epsilon$. Moreover, $\check{\tau}$ belongs to the normal line  to $\partial \Omega$ at $\check{x}$, which will be denoted by $L$. From Lemma \ref{L36}, we are able to find $y \in L \cap \Omega_{\lambda_1(\check{x})}^{\nu(\check{x})} $ with $\mbox{dist}(y, \check{x}) = \mbox{dist}(y, \partial \Omega) = \rho - \epsilon$. Since we are supposing that there are no global maximum points of $u_n$ at a distance of $\partial \Omega$ greater  than or equal to $\rho - \epsilon$, we conclude that $u(y) < u(\check{\tau})$, but this fact contradicts the monotonicity established in Lemma \ref{L37}. So the claim is proved. 
\vspace{0.2cm}

In what follows, we employ a blow-up method to derive a contradiction with the existence of the positive bifurcation parameter $\check{\lambda}\in (0, \Lambda^*]$. For this proposal, denote by 
$$M_n = \|u_n\|_\infty = u_n(\tau_n),$$ where $\tau_n$ is a maximum point of $u_n$ given by Claim 1, and define 
$$w_n(y) = \frac{u_n(M_n^{-k}y + \tau_n)}{M_n}, ~~~y \in ~\Omega_n := M_n^k\left(\Omega - \tau_n\right), $$ where $k= (q-p+1)/p > 0$. Then, from the fact that $(\lambda_n, u_n) \in \Sigma_{\epsilon_n}$ and using change of variable in the integral one obtains
$$ \displaystyle\int_{\Omega_n} |\nabla w_n|^{p-2} \nabla w_n \nabla \varphi dy  = \lambda_n \displaystyle\int_{\Omega_n} \left[w_n^q + M_n^{-kp - p+1}(M_nw_n + \epsilon_n)^{-\delta}\right]\varphi dy, ~~~\forall ~\varphi \in C^{\infty}_c(\Omega_n). $$
Given any $R>0$, we obtain from $M_n \rightarrow \infty$ and $k > 0$ that  $\overline{B_R(0,0)} \subset \Omega_n$ for $n$ large enough, where $B_R(0,0)$ is the ball in $\mathbb{R}^n$  centered at the origin  with radius $R$. Fixing a such ball, notice that 
\begin{equation}\label{eq132} \left(w_n(y) M_n + \epsilon_n\right)^{-\delta }
 \leq [u_n(M_n^{-k}y + x_n)]^{-\delta} \leq [\omega_{\lambda_n,1}(M_n^{-k}y + x_n)]^{-\delta}, ~ ~y \in \overline{B_R(0,0) },\end{equation} where $\omega_{\lambda_n,1}$ is the solution of (\ref{sp}) with $\lambda = \lambda_n$ and $\epsilon = 1$. Noting that $\lambda_n \rightarrow \check{\lambda} > 0$, we get from (\ref{eq132}) that $ \left(w_n(y) M_n  + \epsilon_n\right)^{-\delta } \leq C_R$ in $\overline{B_R(0,0)}$, for some $C_R$  depending on $R$ but not of $n$.  In this way, we can apply once again the regularity results of Lieberman \cite{Lieberman}
 to conclude that $w_n$ is $C^{1,\alpha}(\overline{B_R(0,0)})$ uniformly bounded. Hence, using Arzel�-Ascoli theorem and a diagonalization argument one obtains a subsequence which converges locally uniformly in $C^{1,\beta}(\mathbb{R}^n)$ to a $w \in C^1(\mathbb{R}^n)$ satisfying 
$$\left\{\begin{array}{ll} 
\displaystyle\int_{\mathbb{R}^n} |\nabla w|^{p-2}\nabla w \nabla \varphi dx = \check{\lambda}\displaystyle\int_{\mathbb{R}^n} w^q \varphi dx, ~~~\forall \varphi \in C_c^{\infty}(\mathbb{R}^n) \\
\|w\|_\infty =1, ~~~ w > 0 ~\mbox{in} ~\mathbb{R}^n,
\end{array}\right.$$ that contradicts the result of Serrin and Zou (see Theorem II in \cite{Serrin}).
Therefore, $\lambda = 0$ is the only bifurcation value of $\Sigma$ from infinity. As a consequence, if $\lambda_n \nearrow  \Lambda^*$ and $\underline{u}_{\lambda_n}$ is the minimal solution of $(P)$ with $\lambda = \lambda_n$, then $\|\underline{u}_{\lambda_n}\|_\infty$ is uniformly bounded. So, from Lemma \ref{add} we obtain the existence of $u \in W^{1,p}_{\mathrm{loc}}(\Omega) \cap C_0(\overline{\Omega})$ solution of $(P)$ with $\lambda = \Lambda_*$ such that $u_n \rightarrow u$ in $W^{1,p}_{\mathrm{loc}}(\Omega)$.
\fim

\section{Conclusion}

In this paper, we present a new approach to deal with elliptic quasilinear problems perturbed by strongly-singular terms combined with $(p-1)$-superlinear nonlinearities on smooth bounded domains. With this approach, we were able to establish not only 
 $\lambda$-ranging for existence and multiplicity but also qualitative information of the solutions depending on the parameter $\lambda>0$.  However, mainly due to the lack of a priori estimates and strong comparison principle for strongly-singular problems (that in general requires $C^1(\overline{\Omega})$-regularity of the solutions), we only provided a local multiplicity in the strong singular case ($\delta \geq 1$).  An important challenge in this class of problems is to establish conditions to obtain global multiplicity.

\end{document}